%% file: AsynchronousNetworks_I.tex
\documentclass[12pt]{amsart}
\usepackage{graphicx,amssymb,latexsym, amsmath,amsthm}

\usepackage[english]{babel}
\usepackage[latin1]{inputenc}
\usepackage{xy}
\usepackage{graphicx}
\usepackage{verbatim}
\usepackage{enumerate}
\usepackage{url}
\usepackage{subfigure}
\usepackage{float}
\usepackage{multicol}
\usepackage{mathtools}

\newcommand{\examend}{\hfill \mbox{$\diamondsuit$}}
\newcommand{\cx}{\mathbb{C}}
\newcommand{\NS}{N_0}
\newcommand{\CC}[1]{M_\bullet(#1)}

\newcommand{\Ref}[1]{(\ref{#1})}
\newcommand{\real}{\mathbb{R}}
\newcommand{\intg}{\mathbb{Z}}
\newcommand{\XX}{\mathbf{X}}

\newcommand{\xx}{{\mathbf{x}}}

\newcommand{\zz}{{\mathbf{z}}}

\newcommand{\pint}{\mathbb{N}}
\newcommand{\arr}{\rightarrow}
\newcommand{\dd}{\,|\,}
\newcommand{\ECS}{{\boldsymbol{\emptyset}}}
\newcommand{\ZO}{\mbox{$0$\,-$1$ }}

\newtheorem{lemma}{Lemma}[section]

\newtheorem{prop}[lemma]{Proposition}
\newtheorem{thm}[lemma]{Theorem}
\newtheorem{cor}[lemma]{Corollary}
\theoremstyle{definition}
\newtheorem{Def}[lemma]{Definition}
\newtheorem{exam}[lemma]{Example}
\newtheorem{exams}[lemma]{Examples}
\theoremstyle{remark}
\newtheorem{rem}[lemma]{Remark}
\newtheorem{rems}[lemma]{Remarks}

\newcommand{\sm}{\smallsetminus}
\newcommand{\defo}{=}
\newcommand{\defoo}{\stackrel{\mathrm{def}}{=}}

\newcommand{\mc}[1]{\mathcal{#1}}
\newcommand{\A}{\mc{A}}
\newcommand{\Ee}{\mc{E}}

\newcommand{\Mb}{\mathbf{M}}

\newcommand{\is}[1]{\mathbf{#1}}
\newcommand{\bu}[1]{{#1}^{\bullet}}
\newcommand{\iz}[1]{\bu{\is{#1}}}

\newcommand{\Nn}{\mathcal{N}}

\newcommand{\rset}[2]{\left\lbrace #1 \dd #2\right\rbrace}

\newcommand{\set}[2]{\rset{#1}{#2}}

\newcommand{\sset}[1]{\left\lbrace #1\right\rbrace}

\newcommand{\R}{\mathbb{R}}

\newcommand{\Net}{\mathfrak{N}}
\newcommand{\F}{\mathcal{F}}

\newcommand{\f}{\mathbf{f}}

\newcommand{\tb}{\boldsymbol{\tau}}

\newcommand{\ah}{\widehat{\alpha}}
\newcommand{\bh}{\widehat{\beta}}

\newcommand{\ol}[1]{\overline{#1}}

\newcommand{\CSz}[1]{#1^0}
\newcommand{\CSf}[1]{#1^\flat}

\renewcommand{\aa}{\alpha}
\newcommand{\af}{\CSf{\aa}}
\newcommand{\az}{\CSz{\aa}}

\newcommand{\bphi}{\boldsymbol{\phi}}
\newcommand{\bpsi}{\boldsymbol{\psi}}

\title[Asynchronous Networks]{Asynchronous Networks and Event Driven Dynamics} 
\dedicatory{Dedicated to the memory of David Broomhead}
\author{Christian Bick}
\thanks{Research supported in part by NSF Grant DMS-1265253 \& Marie Curie IEF Fellowship (Project 626111)}
\address{Department of Mathematics, Rice University MS-136, 6100 Main St., Houston, 
TX~77005, USA}
\curraddr{Department of Mathematics, University of Exeter, Exeter EX4 4QF, UK}
\email{c.bick@exeter.ac.uk}
\author{Michael Field}
\thanks{Research supported in part by NSF Grant DMS-1265253 \& Marie Curie IIF Fellowship (Project 627590)}
\address{Department of Mathematics, Rice University MS-136, 6100 Main St., Houston, 
TX~77005, USA}
\curraddr{Department of Mathematics, Imperial College, SW7 2AZ, UK}
\email{mikefield@gmail.com, Michael.Field@imperial.ac.uk}
\date{\today}

\begin{document}

\begin{abstract}
Real-world networks in technology, engineering and biology 
often exhibit dynamics that cannot be adequately reproduced using network 
models given by smooth dynamical systems and a fixed network topology. 
Asynchronous networks give a theoretical and conceptual framework for 
the study of network dynamics where nodes can
evolve independently of one another, be constrained, stop, and later 
restart, and where the interaction between different components of the 
network may depend on time, state, and stochastic effects. This 
framework is sufficiently general to encompass a wide range of 
applications ranging from engineering to neuroscience. Typically, dynamics is piecewise smooth
and there are relationships with Filippov systems. 
In the first part of the paper, we give examples of asynchronous networks,
and describe the basic formalism and structure. In the second part, 
we make the notion of a functional asynchronous network
rigorous, discuss the phenomenon of dynamical locks, and present a 
foundational result on the spatiotemporal factorization of the dynamics
for a large class of functional asynchronous networks. 
\end{abstract}

\newcommand{\imagescaling}{0.9}

\maketitle

\let\oldtocsection=\tocsection
\let\oldtocsubsection=\tocsubsection
\let\oldtocsubsubsection=\tocsubsubsection

\renewcommand{\tocsection}[2]{\hspace{-\normalparindent}\oldtocsection{#1}{#2}}
\renewcommand{\tocsubsection}[2]{\hspace{0em}\oldtocsubsection{#1}{#2}}
\renewcommand{\tocsubsubsection}[2]{\hspace{1em}\oldtocsubsubsection{#1}{#2}}

\tableofcontents

\input{sectionI_01.tex}

\input{sectionI_02.tex}

\input sectionI_03.tex

\input{sectionI_04.tex}

\input{sectionI_05.tex}

\input{sectionI_06.tex}

\input{biblioI.tex}
\end{document}

%% file: sectionI_01.tex

\section{Introduction}
In this work we develop a theory of asynchronous networks and event driven dynamics. This theory constitutes an approach to network dynamics 
that takes account of features encountered in networks from modern technology, engineering, and biology, especially neuroscience. 
For these networks dynamics can involve a mix of distributed and decentralized control,
adaptivity, event driven dynamics, switching, varying network topology and hybrid dynamics (continuous and discrete).  
The associated network dynamics will generally only be piecewise smooth, nodes may stop and later restart 
and there may be no intrinsic global time (we give specific examples and definitions later). 
Significantly, many of these networks have a \emph{function}. For example, transportation 
networks bring people and goods from one point to another and neural networks may perform pattern recognition or
computation.

Given the success of network models based on smooth differential equations and methods based on statistical physics, 
thermodynamic formalism and averaging (which typically lead to smooth network dynamics),
it is not unreasonable to ask whether it is \emph{necessary} to incorporate issues such as nonsmoothness in a theory of network dynamics. 
While nonsmooth dynamics is more familiar in engineering than in physics, 
we argue below that ideas from engineering, control and nonsmooth dynamics apply to many classes of network
and that nonsmoothness often cannot be ignored in the analysis of network function.
As part of these introductory comments, we also explain the motivation underlying our work, and
describe one of our main results: the modularization of dynamics theorem.

\subsection*{Temporal averaging}
Consider the analysis of a network where links are added and removed over time. Two extreme cases have been widely considered in the literature. 
If the network topology switches rapidly, relative to the time scale of the phenomenon being considered, 
then we may be able to replace the varying topology by the time-averaged topology\footnote{For example, if the input structure is additive -- see section~\ref{generalities}.}. Providing that the network topology is not
state dependent, the resulting dynamics will typically be smooth. On the other hand, if the topology changes 
slowly enough relative to the time scale of interest, we may regard the topology as constant and again we obtain smooth
network dynamics. Either one of these approaches may be applicable in a system where time scales are clearly separated.

However, in many situations, especially those involving control or close to bifurcation, 
\emph{changes in network topology may play a crucial role in network function} 
and an averaging approach may fail or neglect essential structure. This is well-known for problems in optimized control where 
solutions are typically nonsmooth and averaging gives the \emph{wrong} solutions (for example, in switching problems using thermostats). 
For an example with variable network topology, we cite the effects of changing connection 
structure (transmission line breakdown), or adding/subtracting a microgrid, on a power grid. 
Neither averaging nor the assumption of constant network structure are appropriate tools: we cannot average the problems away. 
Instead, we are forced to engage with an intermediate regime, where nonsmoothness (switching) and control play a crucial role in network function.

\subsection*{Spatial averaging and network evolution}

Much current research on networks is related to the description and understanding of complex systems~\cite{BA,CH,LLK,RMc}.
Roughly speaking, and avoiding a formal definition~\cite{LLK}, we regard a complex system as a large network of nonlinearly interacting dynamical systems where there
are feedback loops, multiple time and/or spatial scales, emergent behaviour, etc.
One established approach to complex networks and systems uses ideas from statistical mechanics and thermodynamic formalism. 
For example, models of complex networks of interconnected neurons can sometimes be described in terms of their information 
processing capability and entropy~\cite{RNL}.
These methods originate from applications to large interacting systems of particles in physics.
As Schr\"odinger points out in his
1943 Trinity College, Dublin, lectures~\cite{Sc}

\vspace*{0.06in}
\centerline{``...the laws of physics and chemistry are statistical throughout.''}
\vspace*{0.06in}

In contrast to the laws of physics and chemistry, evolution plays a decisive role in the development 
of complex biological structure. Functional biological structures that provided the basis for evolutionary development can be quite small -- the nematode worm \emph{caenorhabditis elegans} has 302 neurons. If the underlying small-scale structure still has functional relevance, an approach based on statistical averages to complex biological networks has to be limited; on the one hand, averaging over the entire network will likely ignore any small scale structure, and on the other hand statistical averages have little or no meaning for small systems -- \emph{at least on a short time scale}.

Reverse engineering large biological structures
appears completely impractical; in part this is because of the role that evolution plays in the development of complex structure. \emph{Evolution works
towards optimization of function, rather than simplicity}, and is often local in character with the flavour of decentralized control.
Similar issues arise in understanding evolved engineering structures. For example,
the internal combustion engine of a car in 1950 was a simple device, whose operation was synchronized through mechanical means.
A modern internal combustion engine is structurally complex and employs a mix of
synchronous and asynchronous systems controlled by multiple computer processors, sensors and complex computer code.

\subsection*{Reductionism}
In nonlinear network dynamics, and complex systems generally, there is the
question as to how far one can make use of reductionist techniques~\cite{Anderson1972}, \cite[2.5]{LLK}.  
One approach, advanced by Alon and Kastan~\cite{KU} in biology, 
has been the identification and description of relatively simple and small dynamical units, 
such as non-linear oscillators or network motifs (small network configurations that occur 
frequently in large biological networks~\cite[Chapter 19]{CH}). Their premise is that a
modular, or engineering, approach to network dynamics is feasible: identify building blocks,  
connect together to form networks and then describe
dynamical properties of the resulting network in terms of the dynamics of its components.
\begin{quote}
``Ideally, we would like to understand the dynamics of the entire network based on the dynamics of the individual building blocks." Alon~\cite[page 27]{Alon}.
\end{quote}
While such a reductionist approach works well in linear systems theory, where a superposition principle holds, or
in the study of synchronization in weakly coupled systems of nonlinear approximately identical oscillators~\cite{PC,PN,BBH,GST}, it is
usually unrealistic in the study of heterogenous networks modelled by a system of analytic nonlinear differential
equations: network dynamics may bear little or no relationship to the intrinsic (uncoupled) dynamics of nodes.

\subsection*{A theory of asynchronous networks.}

The theory of asynchronous networks we develop provides an approach to 
the analysis of dynamics and function in complex networks. 
We illustrate the setting for our main result with a simple example. Figure~\ref{schem1} shows the schematics of a network where there is only intermittent 
connection between nodes\footnote{Figure~\ref{schem1} can be viewed as representing part of a threaded computer program.
The events~$\is{E}^a,\dotsc,\is{E}^h$ will represent synchronization events -- evolution of associated threads is stopped until each thread has finished its computation and then variables are
synchronized across the threads.}.
We assume eight nodes $N_1, \dotsc, N_8$. Each node~$N_i$ will be given an initial state 
and started at time $T_i \ge 0$. Crucially, we assume the network has a function: reaching designated  
terminal states in finite time -- indicated on the right hand side of the figure.
Nodes interact depending on their state. For example, referring to figure~\ref{schem1}, nodes $N_1$, $N_2$ first interact
during the event indicated by~$\is{E}^a$.
Observe there is no
global time defined for this system but there is a partially ordered temporal structure: event~$\is{E}^c$ always occurs after
event~$\is{E}^a$ but may occur before or after event~$\is{E}^b$. We caution that while the direction of time is from left-to-right, 
there is no requirement of moving from left to right in the spatial variables: the phase space dimension for nodes could be 
greater than one and the initialization and terminations sets could be the same. This example can be generalized to allow for 
changes in the number and type of nodes after each event.  
The intermittent connection structure we use may be viewed as an extension of the 
idea of \emph{conditional action} as defined by Holland in the context of complex adaptive systems~\cite{Holl}.

\begin{figure}[h]
\centering
\includegraphics[width=0.8 \textwidth]{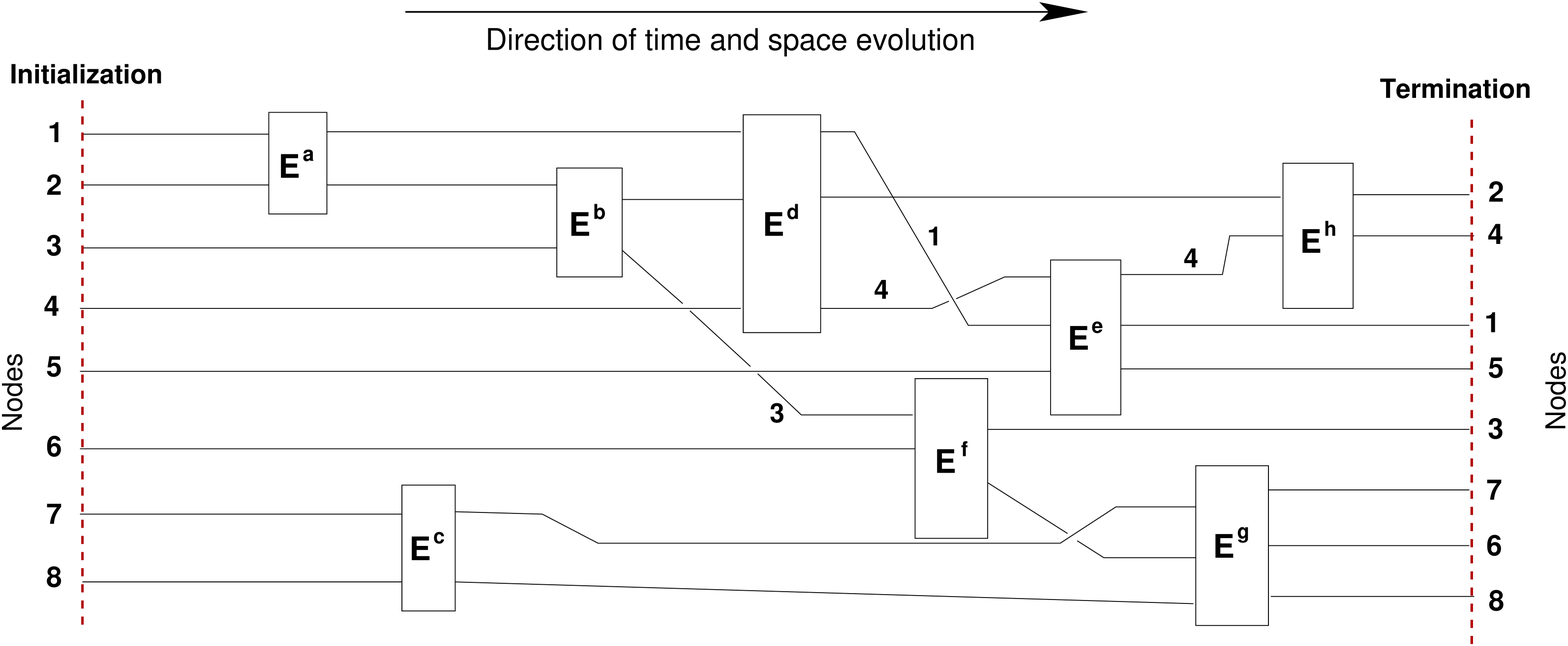}
 
\caption{A functional feedforward network with 8 nodes}
\label{schem1}
\end{figure}

Our main result, stated and proved in part II of this work~\cite{BF1}, is a \emph{modularization of dynamics theorem} that yields a functional decomposition for a
large class of asynchronous networks. 
Specifically, we give general conditions that enable us to represent a large class of 
functional asynchronous networks as feedforward functional networks of the type illustrated in figure~\ref{schem1}. As a consequence, the
function of the original network can be expressed explicitly in terms of uncoupled node dynamics and event function. 
Nonsmooth effects, such as changes in network topology through decoupling of nodes and stopping and restarting of nodes, are one of the crucial ingredients 
needed for this result. In networks modelled by smooth dynamical systems,
all nodes are effectively coupled to each other at all times and information propagates instantly across the entire network. Thus, a spatiotemporal decomposition 
is only possible if the network dynamics is nonsmooth and (subsets of) nodes are allowed to evolve independently of each other for periods of time. 
This allows the identification of dynamical units, each with its own function, that together comprise the dynamics and function of the entire network. 
The result highlights a drawback of averaging over a network: the loss of information 
about the individual functional units, and their temporal relations, that yield network function.

A functional decomposition is natural from an evolutionary point of view: the goal of an evolutionary process is optimization of 
(network) function. Thus, rather than asking how network dynamics can be understood in terms of the dynamics of constituent 
subnetworks -- the classical reductionist question -- the issue is how network function can be understood in terms of the 
function of network  constituents. Our result not only gives a satisfactory answer to Alon's question for a 
large class of functional asynchronous networks but suggests an approach to determining key structural features of 
components of a complex system that is partly based on an evolutionary model for development of structure.
Starting with a small well understood model, such as the class of functional feedforward networks described above, 
we propose looking at bifurcation in the context of optimising a network function -- for example, understanding the effect on function
when we break the feedforward structure by adding feedback loops. 

\subsection*{Relations with distributed networks}
An underlying theme and guide for our formulation and theory of asynchronous networks is that of efficiency and cost in large distributed networks.
We recall the guidelines given by Tannenbaum \& van Steen~\cite[page 11]{TvS} for scalability in large distributed networks 
(italicised comments added):
{\small 
\begin{itemize}
\item No machine has complete information about the (overall) system state. [\emph{communication limited}]
\item Machines make decisions based only on local information. [\emph{decentralized control}]
\item Failure of one machine does not ruin the algorithm. [\emph{redundancy}]
\item There is no implicit assumption of global time.
\end{itemize}
} 
Of course, networks dynamics, in either technology, engineering or biology, is likely to involve a complex mix of synchronous and asynchronous components.
In particular, timing (clocks, whether local or global) may be used to trigger the onset of events or processes as part of a weak mechanism for 
centralized control or resetting. Evolution is opportunistic -- whatever works well will be adopted (and adapted) whether synchronous or asynchronous in character.
In specific cases, especially in biology, it may be a matter of debate as to which viewpoint -- synchronous or asynchronous -- is the most appropriate.
The framework we develop is sufficiently flexible to allow for a wide mix of synchronous and asynchronous structure at the global or local level.

\subsection*{Past work}
Mathematically speaking, much of what we say has significant overlap with 
other areas and past work. We cite in particular, the general area of nonsmooth dynamics,
Filippov systems and hybrid systems (for example, \cite{Fi,As,Min,BBCK})
and time dependent network structures (for example, \cite{BBKP,LuAtayJost,GBG,Holme}). While the theory of nonsmooth dynamics 
focuses on problems in control, impact, and engineering,
rather than networks, there is significant work studying bifurcation (for example~\cite{KRG,bbc}) which is likely to apply
to parts of the theory we describe. 
From a vast literature on networks and dynamics, we cite Newman's text~\cite{Newman1} for a comprehensive 
introduction to networks, and the very recent tutorial of Porter \& Gleeson~\cite{PG} which addresses questions related to our work,
gives an overview and introduction to dynamics on networks, and includes an extensive bibliography of past work.

\subsection*{Outline of contents}
After preliminaries in section~\ref{generalities}, 
we give in section~\ref{asyncsec} vignettes (no technical details) of several asynchronous networks from technology, engineering, 
transport and neuroscience.  
In section~\ref{sec:AsynNetModel}, we give a mathematical formulation of an asynchronous network with a focus on event driven dynamics,
and constraints. We follow in section~\ref{ampex} with two more detailed examples of asynchronous networks 
including an illuminating and simple example of
a transport network which requires minimal technical background yet exhibits many characteristic features of an asynchronous network, 
and a discussion of power grid models that indicates both the limitations and possibilities of our approach.
We conclude with a discussion of products of asynchronous networks in section~\ref{sec:Products} that illuminates some of the
subtle features of the event map. In part II~\cite{BF1}, we develop the
theory of functional asynchronous networks and give the statement and proof of the modularization of dynamics theorem.

\medskip
\noindent \emph{Dedication.}
The genesis of this paper lies in a visit in 2010 by one us (MF) to work with Dave Broomhead at Manchester University.
Dave was very interested in asynchronous processes and local clocks. During the visit, he came up with a 2 cell
random dynamical systems model for the investigation of asynchronous dynamics and local time.  This 2 cell model provided the
seed and stimulus for the work described in this paper.  Dave's illness and untimely death
sadly meant that our planned collaboration on this work could not be realized.

%% file: sectionI_02.tex
\section{Preliminaries and generalities on networks}
\label{generalities}
\subsection{Notational conventions}
We recall a few mostly standard notational conventions used throughout.
Let $\pint$ denote the natural numbers (the strictly positive integers),
$\intg_+$ denote the set of nonnegative integers,
$\real_+ = \set{x \in \real }{ x \ge 0}$, and
$\real(>0) = \set{x \in \real_+ }{ x \ne 0 }$. Given
$n \in \pint$, define $\is{n} \defo \sset{1,\dotsc,n}$. Let
$\iz{n} \defo \sset{0, 1, \dotsc, n}$ and, more generally, for
$A\subset\pint$ define $\bu{A}=A\cup\sset{0}$.

\subsection{Network notation} 
We establish our conventions on network notation; we  follow these throughout this work.

Suppose the network $\Nn$ has $k$ nodes, $N_1,\dotsc,N_k$. Abusing notation, we often let $\Nn$ denote both network and the set of nodes $\{N_1,\dotsc,N_k\}$.
Denote the state or phase space for $N_i$ by $M_i$\footnote{We assume the phase space for each node is a connected differential manifold -- usually a domain in $\real^n$ or
the $n$-torus, $\mathbb{T}^n$.} and set $\Mb = \prod_{i \in \is{k}} M_i$ -- the network phase space. 
Denote the state of node $N_i$ by $\xx_i \in M_i$ and the
network state by $\XX = (\xx_1,\dotsc,\xx_k) \in \Mb$.  

Smooth dynamics on $\Nn$
will be given by a system of ordinary differential equations (ODEs) of the form
\begin{equation}
\label{EQ1}
\xx_i' = f_i(\xx_i;\xx_{j_1},\dotsc, \xx_{j_{e_i}}),\;\; i \in \is{k},
\end{equation}
where the components $f_i$ are at least $C^1$ (usually $C^\infty$ or analytic) and
the following conditions are satisfied. \\

\noindent (N1) For all $i \in \is{k}$, $j_1 < \dotsc < j_{e_i}$ are distinct
elements of $\is{k} \sm \{i\}$ (and so $e_i < k$). \\

Set $J(i) = \{j_1, \dotsc, j_{e_i}\} \subset \is{k}$, $i \in \is{k}$ ($J(i)$ may be empty). \\

\noindent (N2) For each $i \in \is{k}$, the evolution of $N_i$ depends nontrivially on the state of
$N_j$, $j \in J(i)$, in the sense that there exists a
choice of $\xx_i\in M_i$ and $\xx_{j_s} \in M_{j_s}$, $j_s \in J(i) \sm \{j\}$,
such that $f_i(\xx_i;\xx_{j_1},\dotsc, \xx_{j_{e_i}})$ is not constant as a function of $\xx_j$.  \\

\noindent (N3)  We generally assume the evolution of $N_i$ depends 
on the state of $N_i$. If we need to emphasize that $f_i$ does not depend on $\xx_i$ 
in the sense of (N2), we write
$f_i(\xx_{j_1},\dotsc, \xx_{j_{e_i}})$, if $J(i) \ne \emptyset$. 
If $J(i) = \emptyset$, we regard the dependence of~$f_i$ on~$\xx_i$ as 
nontrivial iff~$f_i$ is not identically zero and then write 
$f_i(\xx_i)$. Otherwise $f_i \equiv 0$.

\begin{rem}
\label{remdep}
Given network equations~\Ref{EQ1} which do not satisfy (N1--3), we can first redefine the $f_i$ so as to satisfy (N1). Next we
remove trivial dependencies so as to satisfy (N2). Finally, we check for the dependence of $f_i$ on the internal state $\xx_i$ and modify the $f_i$ as necessary to achieve (N3). 
If $f_i \equiv 0$, we can remove the node from the network.
Consequently, it is no loss of generality to always assume that (N1--3) are satisfied, with $f_i \not\equiv 0$. A consequence is that any \emph{network vector field}
$\f=(f_1,\dotsc,f_k):\Mb \arr T\Mb$
can be uniquely written in the form~\Ref{EQ1} so as to satisfy (N1--3).
\end{rem}

Let $M(k)$ denote the space of $k\times k$ \ZO matrices $\beta = (\beta_{ij})_{i,j \in \is{k}}$ with coefficients
in $\sset{0, 1}$ and $\beta_{ii} = 0$, all $i \in \is{k}$. 
Each $\beta \in M(k)$ determines uniquely a directed graph $\Gamma_\beta$ with vertices
$N_1,\dotsc,N_k$ and directed edge $N_j \arr N_i$ iff $\beta_{ij} = 1$ and $i \ne j$. The matrix $\beta$ is the
adjacency matrix of $\Gamma_\beta$.  We refer to $\beta$ as a \emph{connection structure} on $\Nn$.

If $\f:\Mb \arr T\Mb$ is a network vector field satisfying (N1--3), then $\f$
determines a unique connection structure $C(\f) \in M(k)$ with associated graph $\Gamma_{C(\f)}$. In order to specify the graph uniquely,
it suffices to specify the set of directed edges.

We define the \emph{network graph} $\Gamma = \Gamma(\Nn,\f)$ to be the directed graph $\Gamma_{C(\f)}$. Thus, $\Gamma(\Nn,\f)$ has node
set $\Nn = \{N_1,\dotsc,N_k\}$ and a directed connection $N_j \arr N_i$ will be an edge of $\Gamma$
if and only if $j \ne i$ and the dynamical evolution of $N_i$ depends nontrivially on the state of $N_j$. 
\begin{rem}
\label{remdep2}
Our conventions are different from formalisms involving
multiple edge types (for example, see \cite{GST,AADF} for continuous dynamics and \cite{AF1} for discrete dynamics).
We allow at most one connection between distinct nodes of the network graph and do not use self-loops: connections encode dependence.
\end{rem}

\subsubsection{Additive input structure}

In many cases of interest, we have an additive input structure~\cite{F2014} and the components $f_i$ of $\f$ may be written
\begin{equation}
\label{EQais}
f_i(\xx_i;\xx_{j_1},\dotsc, \xx_{j_{e_i}}) = F_i(\xx_i) + \sum_{s=1}^{e_i} F_{ij_s}(\xx_{j_s} , \xx_i), 
\;i \in \is{k}.
\end{equation}
Additive input structure implies that there are no interactions between
inputs $N_j, N_k \arr N_i$, as long as $j,k \ne i$, $j \ne k$, and allows us to add and subtract inputs and nodes in a consistent way.
We may think of $\xx_i' = F_i(\xx_i)$ as defining the \emph{intrinsic dynamics} of the node.
\begin{rems}
\label{addrems}
(1) Additive input structure is usually assumed for modelling weakly coupled nonlinear oscillators and is required
for reduction to the standard Kuramoto phase oscillator model~\cite{K1984,EK,HI}.\\
(2) If we identify a null state $\zz_j^\star$ for each node~$N_j$, 
then the decomposition \Ref{EQais} will be unique if we require
$F_{ij}(\zz_j^\star,\xx_i) \equiv 0$\footnote{For
identical phase spaces, assume inputs are asymmetric -- $F_{ij}\ne F_{i\ell}$, if $j \ne \ell$.
For symmetric inputs see~\cite{GST,AF2}.}.
If a node is in the null state then it has no output to other 
nodes and is `invisible' to the rest of the network.  If 
we have an additive structure on the phase spaces $M_i$ (for example, 
each~$M_i$ is a domain in~$\real^n$ or an $n$-torus $\mathbb{T}^n$)
it is natural to take $\zz_i^\star = 0$. \\
(3)  If $M_i = \real^n$ or $\mathbb{T}^n$, $i \in \is{k}$, and 
$F_{ij_s}(\xx_{j_s}, 
\xx_i) = G_{ij_s}(\xx_{j_s} - \xx_i)$, 
$i,j_s\in\is{k}$, the coupling is \emph{diffusive} (see~\cite[\S 2.5]{AF1} for general 
phase spaces). 
\end{rems}

\subsection{Synchronous networks}
Systems of ordinary differential equations like \Ref{EQ1}
give mathematical models for \emph{synchronous} networks. By synchronous, we mean 
nodes are all synchronized to a global clock -- the terminology comes from computer science. Indeed,
if each node comes with a local clock, then all the clocks can be set to the same time provided
that the network is connected (we ignore the issue of delays, but see~\cite{LL}). 
The synchronization of local
clocks is essentially forced by the model and the connectivity of the network graph; nodes cannot evolve independently of one another 
unless the network is disconnected. 

We recall some characteristic features of
synchronous networks.

\begin{description}
\item[Global evolution]
Nodes \emph{never} evolve independently of each other: if the state of any node is perturbed, then generically the
evolution of the states of the remaining nodes changes. 
\item[Stopped nodes]
If a node (or subset of node variables) is at equilibrium or ``stopped'' for a period of time, it will remain stopped for all future time.
If a node has a non-zero initialization, it will never stop (in finite time).
\item[Fixed connection structure]
The connection structure of a synchronous network is fixed:
it does not vary in time and is not dependent on node states:   one system of ODEs suffices to
model network dynamics.
\item[Reversibility] Solutions are uniquely defined in backward time.
\end{description}

%% file: sectionI_03.tex
\section{Asynchronous networks: examples}
\label{asyncsec}
In this section, we give several vignettes of asynchronous networks that illustrate
the main features differentiating them from synchronous networks. We amplify
two of these examples in section~\ref{ampex} after we have developed our 
basic formalism for asynchronous networks.

\begin{exam}[Threaded and parallel computation]
\label{thread}
Threaded or parallelized computation provides an example of a discrete stochastic asynchronous network.  
Computation based on a single processor (or single core of a processor) proceeds synchronously and sequentially. 
The speed of the computation is dependent on the clock speed of the processor as the processor clock synchronizes the various steps in the computation.
In threaded or parallel computation, computation is broken into
blocks or `threads' which are then computed \emph{independently} of each other at a rate that is partly dependent on the clock rates of the
processors involved in the computation (these need not be identical).  At certain points in the computation, 
threads need to exchange information with other threads. This process 
involves stopping and synchronizing (updating) the thread states: a thread may have to stop and wait for other threads to complete their
computations and update data before it can continue with its own computation.  

Threaded computation is  non-deterministic:  the running and stopping times of
each thread are unpredictable and differ from run to run. 

Each thread has its own clock (determined by its associated processor).  Threads will be unaware of
the clock times of other threads except during the stopping and synchronization events which can be managed
synchronously (central control) or asynchronously (local control).

This example shows many characteristic features of an asynchronous network: nodes (threads) evolving independently of each other,
and stopping, synchronization  and restarting events. The network also has a \emph{function} -- transforming a set of initial data into a set
of final data in finite time -- and there is the possibility of incorrect code that can lead to 
a process that stops before the computation is complete (a deadlock), or errors where threads try to access a resource at the same time (race condition). \examend
\end{exam}

\begin{exam}[Power grids \& microgrids]
\label{pg}

A power grid consists of a connected network of various types of generators and loads connected by transmission lines.
A critical issue for the stability of the power grid is maintaining tight voltage frequency synchronization across the 
grid in the presence of voltage phase differences between generators and loads and variation in generator outputs and loads. 
We refer to Kundur~\cite{Ku} for classical power grid theory, D\"orfler \emph{et al.}~\cite{DCB} or Nishikawa \& Motter~\cite{Motter2015}, for some more recent and
mathematical perspectives, and \cite{Ku2} for general issues and definitions on power system stability.

Historically, power grids have been centrally controlled and one of the main stability issues has been the effect on
stability of a sudden change in structure -- such as the removal of a transmission line, breakdown of a generator or big change in load.
Detailed models of the power grid need to account for a complex multi-timescale stiff system. Typically stability has been analyzed using numerical methods.
However, relatively simple classes of network models for power grids based on frequency and phase synchronization
have been developed which are applicable for the analysis of some stability and control issues, especially those described in the next paragraph. 
We describe these models in more technical detail in section~\ref{ampex}.

Interest has recently focused on renewable (small) energy sources in a power grid (for example, wind and solar power) and
how to integrate microgrids based on renewable sources into the power grid using a mix of centralized and decentralized control. 
Concurrent with this interest is the issue of smart grids: modifying local loads in terms of the availability and real time costs of power.
While the classical power grid model is of a synchronous network, though with asynchronous features such as the effects on stability of 
the breakdown of a connection (transmission line), these problems focus on asynchronous networks. For example, given a microgrid
with renewable energy sources such as wind and solar, time varying loads and buffers (large capacity batteries), how can the microgrid
be switched in and out of the main power grid while maintaining overall system stability? In this case, switching will be determined by
state (for example, frequency changes in the main power grid signifying changes in power demand or changes in the output of renewable sources or battery reserves) 
and stochastic effects (resulting, for example, from load changes and the incorporation of smart grid technology). 
This is already a tricky problem of distributed and decentralized control with just one microgrid; in the presence of many microgrids 
there is the potential problem of synchronization of switching microgrids in and out of the main power grid. Similar problems occur in smart grids~\cite{T}.

Asynchronous features of power grid networks include variation in connection and node structure (separation, or islanding, of microgrids from main power grid),
state dependence of connection structure, synchronization and restarting events (during incorporation of microgrid into main grid).
\examend
\end{exam}

\begin{exam}[Thresholds, spiking networks and adaptation]
\label{spike}
Many mathematical models from engineering and biology incorporate thresholds. For networks,
when a node attains a threshold, there are often changes (addition, deletion, weights) in connections to another nodes. 
For networks of neurons, reaching a threshold can result in a neuron firing (spiking) and short term connections to
other neurons (for transmission of the spike).  For learning mechanisms, such as Spike-Timing Dependent Plasticity (STDP)~\cite{Ge}
relative timings (the order of firing) are crucial~\cite{GK,CD,MDG} and so each neuron, or connection between a pair of neurons, comes
with a `local clock' that governs the adaptation in STDP.  In general, networks with thresholds, spiking and adaptation
provide characteristic examples of asynchronous networks where dynamics is piecewise smooth and hybrid -- a mix of
continuous and discrete dynamics. Spiking networks also highlight the importance of
efficient communication in large networks: spiking induced connections between neurons are brief and low cost. 
There is also no oscillator clock governing all computations along the lines of a single processor computer. 
These examples all fit well into the framework of asynchronous networks but, on account of the background knowledge required,
we develop the theory and formalism elsewhere~\cite{BF2}. \examend
\end{exam}

\begin{exam}[Transport \& production networks]
\label{pass1}
We discuss transport networks first. For simplicity, we work with a single transport mode: trains. Typically, trains have to be scheduled to be
in a station for overlapping times (stopping, restarting, connections and local times) so that passengers can transfer between trains, or stop in a passing loop (so that trains can pass on a 
single track line). In addition, a train can divide into two parts or two trains can be combined (variation in node structure, stopping and synchronization event).
Generally, transport networks will have asynchronous features and exhibit state dependent connection structure, local times and have a strong stochastic component (for 
example, in stopping and restarting times). We develop a simple formal transport model in section~\ref{ampex} (and in part II~\cite{BF1})
that illustrates basic ideas and results in the theory of asynchronous networks but does not require extensive background knowledge.

A simple example of a production
network is a paint mixer. Assume a
controller which accepts inputs -- requested colour -- which, after computation
to find tint weights (`tint code'), signals a request to inject
selected tints according to the tint code into the base paint 
which is then mixed. The output is a can of coloured and
fully mixed paint. Dynamics plays a limited role -- 
except possibly at the mixing stage (for example, if there is a sensor
that can detect an acceptable level of mixing).
For this network, there is a varying connection
structure determined by the signalling and tint injection. 
A characteristic feature
of this, and many production networks, is the large variation in time
scales. Signalling will typically be very fast, injection moderately fast
and mixing rather slow.
If the times of inputs to the controller are stochastic (for example,
follow a Poisson process), then there will be issues of queueing and prioritization
of inputs. If it is intended to maximize usage of the production
facilities and avoid long waits, then it is natural to suppose that there
are several mixing units and the output of the tint units is switched
between mixer units according to their availability.
Of course, the paint mixer may be a small part of a much
larger distributed production network for which we can expect multiple time scales,
switching between production units, changing the output of production
units, stopping or restarting units, etc.
The control of large distributed production systems will typically
involve a mix of decentralized and centralized control. 

Synthesis of proteins at the cellular level can be viewed as a generalization
of the paint mixer model.  We refer the reader to Alon~\cite[8, Chapter 1]{Alon} for background and
more details, especially on transcription networks.
\examend
\end{exam}

We summarize some of the key features of asynchronous networks illustrated by all of the preceeding examples.
\begin{enumerate}
\item Variable connection structure and dependencies between nodes. Changes in connection structure may depend on the state of the system
or be given by a stochastic process.  
\item Synchronization events associated with stopping or waiting states of nodes. 
\item Order of events may depend on the initialization of the system or stochastic effects.  
\item Dynamics is only piecewise smooth and there may be a mix of continuous and discrete dynamics. 
\item Aspects involving function, adaptation and control. 
\item Evolution only defined for forward time -- systems are non-reversible. 
\end{enumerate}

%% file: sectionI_04.tex
\section{A Mathematical model for asynchronous networks}
\label{sec:AsynNetModel}

In this section we formalize the notion of an asynchronous network. 
Our focus is on deterministic (not stochastic)
and continuous time asynchronous networks which are autonomous (no explicit dependencies on time)
and we use the term `asynchronous network' as synonym for a deterministic and autonomous continuous time asynchronous network.

\subsection{Basic formalism for asynchronous networks}
\label{oview}
Consider a 
network $\Nn$ with $k$ nodes, $N_1,\dotsc,N_k$, and follow the conventions of section~\ref{generalities}:
each node $N_i$ has phase space $M_i$,  and 
$\Mb = \prod_{i = 1}^k M_i$ -- the network phase space.  A network vector field $\f$ on  $\Mb$ is
assumed to satisfy conditions (N1--3) and so 
determines a unique connection structure $C(\f) \in M(k)$
and associated network graph $\Gamma_{C(\f)}$ (no self-loops).

Stopping, waiting, and synchronization are characteristic features of
asynchronous networks. If nodes of a network are stopped or partially
stopped, then node dynamics will be constrained to subsets of node phase
space. We codify this situation by introducing a \emph{constraining node}
$\NS$ that, when connected to $N_i$, implies that dynamics on $N_i$
is constrained. We give the precise definition of constraint shortly (in 4.3); for the present, the reader 
may regard a constrained node as stopped -- node dynamics is defined by the zero vector field.
We only allow connections $\NS \arr N_i$, $i\in \is{k}$, and do not 
consider connections $N_i \arr \NS$, $i \in \is{k}^\bullet$. 
Henceforth we usually always assume there is a constraining node
and let $\Nn = \sset{\NS, N_1,\dotsc,N_k}$ denote the set of nodes. 
We emphasize that the constraining node $\NS$ has no dynamics and no associated phase space.
In a network with no constraints (there are no connections $\NS\arr N_i$), 
the constraining node~$\NS$ plays no role and can be omitted. If we allow constraints, 
there may be more than one type of constraint on a node~$N_i$.   

Suppose that there are $p_i \in \pint$ different constraints on the node $N_i$, $i \in \is{k}$.
Set $\is{P} = (p_1,\dotsc,p_k) \in \intg_+^k$ and let $\CC{k;\is{P}}$
denote the space of $k \times (k+1)$ matrices
$\aa = (\aa_{ij})_{i\in \is{k}, j \in \iz{k}}$ such that
\begin{enumerate}
\item $(\aa_{ij})_{i,j \in \is{k}} \in M(k)$ (and so $\alpha_{ii} = 0, i \in \is{k}$).
\item $\aa_{i0} \in \is{p}_i^\bullet$, $i \in \is{k}$.
\end{enumerate}
If $\alpha \in \CC{k;\is{P}}$, we define the directed graph $\Gamma_\alpha$ by
\begin{enumerate}
\item $\Gamma_\alpha$ has node set $\Nn$.
\item For all $i,j \in \is{k}$, $N_j \to N_i$ is an edge iff $\aa_{ij}=1$.
\item $\NS \arr N_i$ is an edge iff $\aa_{i0} \ne 0$. We write $\NS \stackrel{\ell}{\arr} N_i$ if we
need to specify the constraint corresponding to $\ell \in \is{p}_i$.
\end{enumerate}
We usually abbreviate $\CC{k;\is{P}}$ to $\CC{k}$.
Let $\ECS\in\CC{k}$ denote the empty connection structure (no edges).

If $\aa\in \CC{k}$, let $\az$ denote the first column $(\aa_{i0})_{i \in \is{k}}$ of $\aa$. We have a natural projection
$\pi: \CC{k} \arr M(k)$; $\aa \mapsto \af$, defined by omitting the column $\az$. We write $\aa \in \CC{k}$ uniquely as
\[
\aa = (\az \dd \af).
\]
The column vector $\az$ codifies the connections from the constraining node and
$\af$ encodes the connections between the nodes $\{N_1,\dotsc, N_k\}$.

Let $\aa \in \CC{k}$. We provisionally define an \emph{$\aa$-admissible vector field} $\f=(f_1,\dotsc,f_k)$
to be a network vector field such that for $i,j \in \is{k}$, $i \ne j$, $f_i$ 
depends on the state $\xx_j$ of $N_j$ iff~$\aa_{ij} = 1$.
If there is a connection $\NS \arr N_i$ ($\alpha_{i0} \ne 0$), then there is a nontrivial constraint 
on~$N_i$.
An $\aa$-admissible vector field has constrained dynamics if there are
connections from the constraining node.   If $\alpha = \ECS$, nodes are uncoupled and unconstrained.

\begin{Def}
(Notation and assumptions as above.)
\begin{enumerate}
\item A \emph{generalized connection structure}
$\A$ is a (nonempty) set of connection structures on $\Nn$. 
\item An  \emph{$\A$-structure $\mathcal{F}$} is a set
$\mathcal{F} = \{\f^\aa \dd \aa \in \A\}$ of network vector fields such that
each $\f^\aa\in \mathcal{F}$ is $\aa$-admissible.
\end{enumerate}
\end{Def}

Interactions between nodes in asynchronous networks may vary 
and can be state or time dependent or both. 
We focus on state dependence and assume interactions and 
constraints are determined by the state of the network through an 
\emph{event map} $\Ee: \Mb \arr \A$.

\begin{Def}
Given a network $\mathcal{N}$, generalized connection structure 
$\A$, $\A$-structure $\mathcal{F}$, and surjective event map $\mathcal{E}:\Mb \arr \A$,
the quadruple $\Net = (\mathcal{N},\mathcal{A},\mathcal{F},\mathcal{E})$
defines an \emph{asynchronous network}. 

The network vector field of $\Net$ is
given by the state dependent vector field $\is{F}:\Mb \arr T\Mb$ defined by
\[
\is{F}(\XX) = \f^{\mathcal{E}(\XX)}(\XX), \; \XX \in \Mb.
\]
\end{Def}
\begin{rems}
(1) Subject to simple regularity conditions, which we give later, the network vector field $\is{F}$
will have a uniquely defined semiflow.\\
(2) In the sequel we often use the notation $\Net$ as shorthand for the asynchronous network 
$(\Nn,\A,\F,\Ee)$ (by extension, $\Net^a$ will be shorthand for $(\Nn^a,\A^a,\F^a,\Ee^a)$, etc.).
\end{rems}

\begin{exam}
\label{conex}
Let $k = 2$ and $M_1 = M_2 = \real\times \mathbb{T}$. Suppose that dynamics of the uncoupled node $N_i$
is given by the smooth vector field $V_i(x_i,\theta_i) = (f_i(x_i),\omega_i)$, where $f_i(0)\ne 0$, $\omega_i \in \real$, $i \in \is{2}$.

Assume constrained dynamics for either node is defined on the invariant circle $\{0\} \times \mathbb{T}\subset \real\times \mathbb{T}$ by
the vector field $Z_i(x_i,\theta_i) = (0,\omega_i)$, $i \in \is{2}$.
When both nodes  are constrained ($x_1 =x_2 = 0$), assume (constrained) coupling is defined by the vector field $H = (H_1,H_2)$, where
\begin{eqnarray*}
H_1(x_1,\theta_1,x_2,\theta_2)&=&(0,\omega_1+h(\theta_2-\theta_1)) \\
H_2(x_1,\theta_1,x_2,\theta_2)&=&(0,\omega_2+h(\theta_1-\theta_2)), 
\end{eqnarray*}
and $h:\mathbb{T} \arr \real$ is smooth.
The 2-tori $\{(x_1,x_2)\} \times \mathbb{T}^2$ are invariant by the flow of $H$
for all $(x_1,x_2)\in \real^2$.
Revert to standard (uncoupled and unconstrained) dynamics when $|\theta_1-\theta_2| \le \varepsilon$, where $0 < \varepsilon \ll 1$.
We describe the network dynamics using asynchronous network formalism.

Take the generalized connection structure $\A = \{\ECS,\alpha_1,\alpha_2,\beta\}$, where
$\alpha_i=\NS \arr N_i$, $i \in \is{2}$, and 
$\beta=\NS \arr N_1 \leftrightarrow N_2 \leftarrow \NS$.

Take $\F = \{\f^\gamma \dd \gamma \in \A\}$, where
\[
\f^\ECS  =  (V_1,V_2),\;
\f^{\alpha_1}  =  (Z_1,V_2),\;
\f^{\alpha_2}  =  (V_1,Z_2),\;
\f^{\beta}  =  (H_1,H_2).
\]
Define the event map $\Ee: (\real\times\mathbb{T})^2 \arr \A$ by
\begin{eqnarray*}
\Ee(0,\theta_1,0,\theta_2) & = & \beta,\;\text{if } |\theta_1-\theta_2| > \varepsilon\\
& = & \ECS, \;\text{if } |\theta_1-\theta_2| \le \varepsilon\\
\Ee(0,\theta_1,x_2,\theta_2) & = & \alpha_1,\;\text{if } x_2 \ne 0\\
\Ee(x_1,\theta_1,0,\theta_2) & = & \alpha_2,\;\text{if } x_1 \ne 0\\
\Ee(x_1,\theta_1,x_2,\theta_2) & = & \ECS,\;\text{if } x_1x_2 \ne 0.
\end{eqnarray*}
Network dynamics is given by the vector field $\is{F}(\is{X}) = \f^{\Ee(\is{X})}(\is{X})$.
Trajectories for $\is{F}$ are built from pieces of the trajectories of $\f^\ECS$, $\f^{\alpha_1}$, $\f^{\alpha_2}$, and $\f^{\beta}$. 
Using the condition $f_i(0) \ne 0$, $i \in \is{2}$,
we see easily that $\is{F}$ has a well-defined semiflow $\Phi_t(x_1,\theta_1,x_2,\theta_2)$, 
which is continuous in time $t \ge 0$ but is not necessarily continuous in $(x_1,\theta_1,x_2,\theta_2)$.
\examend
\end{exam}

\subsection{Local foliations}
\label{constraint:sec}
Conditions for a constrained node $N_i$ will be given in terms of \emph{foliations} of open subsets of $M_i$. We
start by recalling basic definitions on foliations (see~\cite{BL} for a detailed review).

A $p$-dimensional
smooth (always $C^\infty$ here) foliation $\mathcal{L}$ of the $m$-dimensional manifold $W$
consists of a partition $\{L_\aa \dd \aa \in \Lambda\}$ of $W$ into
connected sets, called \emph{leaves}, such that for every $x \in W$,
we can choose an open neighbourhood $U$ of $x$ and smooth embedding
$\psi:U \arr \real^m$ such that for each leaf $L_\aa$, the components
of $\phi(L_\aa \cap U)$ are given by equations $x^{p+1} = \text{constant}, 
\ldots,  x^m = \text{constant}$.
Each leaf of a foliation will be an immersed $p$-dimensional submanifold
of $W$.  For our applications, we always assume leaves are properly embedded
closed submanifolds of~$W$,  $p < m$, and
that the manifold $W$ has \emph{finitely} many connected components.
In general, a smooth foliation of the  manifold $W$
will consist of a smooth foliation of each connected component of $W$ 
such that the dimension of leaves is constant on each connected component of $W$.

\begin{exams}
\label{folexam}
(1) Every smooth nonsingular vector field on $W$ defines a 1-dimensional
smooth foliation of $W$ (``flow-box'' theorem of dynamical systems). The
leaves are trajectories of the vector field. \\
(2) If $W = A \times B$, where $A$ and $B$ are manifolds, we have the product foliations $\mathcal{L}(A)$ and $\mathcal{L}(B)$ of $W$ defined by
$\mathcal{L}(A) = \{A \times \{b\} \dd b \in B\}$ and $\mathcal{L}(B) = \{ \{a\} \times B\dd a \in A\}$. Each leaf 
$\mathcal{L}(A)$ is transverse to every leaf of $\mathcal{L}(B)$. More generally, foliations $\mathcal{L}, \mathcal{L}'$ are transverse
if leaves are transverse. 
A foliation of $W$, even by compact 1-dimensional leaves, need not have a transverse foliation. The best-known example is
the Hopf fibration which defines a foliation of $S^3$ into circles.
\examend
\end{exams}

Suppose that $\mathcal{L}$ is a  $p$-dimensional smooth foliation of $W$ with leaves
$\set{L_\aa }{ \aa \in \Lambda}$. The \emph{tangent bundle
along the foliation} $\tau: \mathbb{L} \arr W$ is the smooth vector sub-bundle of the tangent bundle
$TW$ of $W$ defined by
\[
\mathbb{L} = \bigcup_{x \in L_\aa,\, \aa \in \Lambda} T_x L_\aa  \subset TW. 
\]
\subsection{Constrained nodes and admissible vector fields}
\label{constrained:nodes}
Following section~\ref{oview}, we assume $\mathcal{N} = \{N_0,N_1,\dotsc,N_k\}$, where the nodes $N_i$ have phase space $M_i$, $i \in \is{k}$.
Fix a $k$-tuple $\is{P} = (p_1,\dotsc,p_k) \in \intg_+^k$. 
In what follows, we assume $\is{P} \ne \is{0}$.

\begin{Def}
(Notation and assumptions as above.) A family $\is{C} = \{(\is{W}_i,\boldsymbol{\mathcal{L}}_i) \dd i \in \is{k}\}$ is a \emph{constraint structure} on 
$\mathcal{N}$ if,  for all $i \in \is{k}$ with $p_i > 0$,
\begin{enumerate}
\item $\is{W}_i = \{W_i^\ell \dd \ell \in \is{p_i}\}$  is a family of nonempty open subsets of $M_i$.
\item $\boldsymbol{\mathcal{L}}_i = \{ \mathcal{L}_i^\ell \dd \ell \in \is{p_i}\}$, where $\mathcal{L}_i^\ell$ is a smooth foliation of $W_i^\ell$.
\end{enumerate} 
\end{Def}

\begin{rems}
\label{constrain:rems}
(1) If $p_i = 0$, there are no constraints on $N_i$. \\
(2) 
If $p_i = 1$, we set $\is{W}_i = (W_i,\mathcal{L}_i)$ and  $\mathcal{L}_i$ is a smooth foliation of the
nonempty open subset $W_i$ of $M_i$.  If we allow the dimension of leaves to vary 
between different connected components, and the families
$\is{W}_i$ to consist of disjoint open subsets of $M_i$, $i \in \is{k}$, then we can reduce to the
case $p_i \le 1$ by taking $W_i = \bigcup_\ell W_i^\ell$ and $\mathcal{L}_i$ to be the 
foliation determined on $W_i$ by $\mathcal{L}_i|W_i^\ell =  \mathcal{L}_i^\ell$, $\ell \in \is{p_i}$. 
For our applications, it is no loss of generality to assume that $\is{W}_i$ always consists of disjoint open subsets of $M_i$, $i \in \is{k}$.
\end{rems}

We can now give a precise definition of an $\aa$-admissible vector field when there are constraints.

\begin{Def}
\label{Cadmiss}
Fix a constraint structure $\is{C} = \{(\is{W}_i,\boldsymbol{\mathcal{L}}_i) \dd i \in \is{k}\}$ on
$\mathcal{N}$ and 
let $\aa \in \CC{k}$. A smooth vector
field $\is{f} = (f_1,\dotsc,f_k)$ on $\Mb$ is an
\emph{$\aa$-admissible vector field} if 
\begin{enumerate}
\item For $i,j \in \is{k}$, $i \ne j$, $f_i$ depends on $\is{x}_j$ iff $\aa_{ij} = 1$.
\item If $\aa_{i0} = \ell > 0 $, then $f_i$ is tangent to the smooth foliation $\mathcal{L}_i^\ell$ at all points of $W_i^\ell \subset M_i$.  Equivalently, $f_i|W_i^\ell$ 
defines a section of~$\mathbb{L}_i^\ell$, the tangent bundle 
along the foliation $\mathcal{L}_i^\ell$.
\end{enumerate}
\end{Def}
\begin{exam}
\label{stop:constraint}
Suppose that $p_i = 1$ and $\aa_{i0} = 1$ so that there is a constraining connection
$\NS \arr N_i$. Let $\f = (f_1,\dotsc,f_k)$ be $\aa$-admissible,
$M_i = \real^\ell$,  and $\mathcal{L}_i$ be an $(\ell-p)$-dimensional foliation of $M_i$ with
leaves given by $x_{r_1} = c_1,\dotsc, x_{r_p} = c_p$. The components
$f_i^{r_1},\dotsc,f_i^{r_p}$ of $f_i=(f_i^1,\dotsc,f_i^\ell)$ will be
identically zero and the node $N_i$ is partially stopped on each leaf.
This is the situation described in example~\ref{conex} 
where the 1-dimensional foliation of $\real\times\mathbb{T}$ is $\{\{x\} \times \mathbb{T} \dd x \in \real\}$.
\examend
\end{exam}

\begin{rem}
Note that if $\NS \arr N_i \leftarrow N_j$, then the coupling from $N_j$ must respect 
constraints on $N_i$ though now of course the dynamics on
a leaf of $\mathcal{L}_i$ will depend on the state of $N_j$. 
\end{rem}

\subsection{The event map}

Let $\A$ be a generalized connection structure with constraint 
structure $\is{C} = \{(\is{W}_i,\boldsymbol{\mathcal{L}}_i) \dd i \in \is{k}\}$.
Let $\Ee: \Mb \arr \A$ be an event map and recall $\Ee$ is always assumed to be surjective. 

For each $\aa \in \A$, define the \emph{event set} $E^\aa \subset \Mb$ by
\[
E^\aa = \{\XX \in \Mb \dd \mathcal{E}(\XX) = \aa\}.
\]
The event sets $\{E^\aa \dd \aa \in \A\}$ partition the network phase space $\Mb$. We require additional conditions on the event map when there are
constraints. These conditions relate the event sets to the constraint structure $\is{C} $ and are required because 
foliations are only locally defined.

Let $\pi_i:\Mb \arr M_i$ denote the projection map onto the phase space of $N_i$, $i \in \is{k}$. Given $i \in \is{k}$,
$\ell \in \is{p_i}$, define
\[
E_i^\ell = \bigcup_{\{\aa \dd \aa_{i0} = \ell\}} \pi_i(E^\aa) \subset M_i.
\]
\begin{Def}
The event map
$\mathcal{E}: \Mb \to \A$ is \emph{constraint regular} if 
for all $i \in \is{k}$, $\ell \in \is{p_i}$, we have
\[
\overline{E_i^\ell} \subset W_i^\ell
\]
\end{Def} 

Henceforth we assume that event maps are constraint regular.
 
\subsection{Asynchronous network with constraints}
\begin{Def}
\label{defasync}
An asynchronous network
$\mathfrak{N}=(\Nn, \A,\mathcal{F},\mathcal{E})$, with constraint structure $\is{C}$,
consists of 
\begin{enumerate}
\item A finite set $\Nn= \{N_0,N_1,\dotsc,N_k\}$ nodes with associated phase spaces $M_i$, $i \in \is{k}$.
\item A generalized connection structure $\A \subset \CC{k}$.
\item An $\A$-structure $\F = \{\f^\alpha \dd \alpha \in \A\}$ consisting of
admissible vector fields.
\item A (constraint regular) event map $\Ee:\Mb \arr \A$.
\end{enumerate}
\end{Def}
\begin{rem}
\label{comprem2}
If $\A$ consists of a single connection structure~$\aa$ (with or without constraints),
then~$\mathcal{F}$ consists of one vector field
$\f=\f^\aa$, with dependencies given by~$\aa$. We
recover a synchronous network with dynamics defined  by~$\f$
and a fixed connection structure.
\end{rem}

\subsection{Network vector field of an asynchronous network}

An asynchronous network $\Net$ 
uniquely determines the \emph{network vector field} $\is{F}$ by
\begin{equation}
\label{Deq}
\is{F}(\XX) = \f^{\mathcal{E}(X)}(\XX),\; \XX \in \Mb.
\end{equation}

\begin{rems}
\label{abc}
(1) We may give a discrete version of definition~\ref{defasync}: each $\f^\alpha$
will be a network map $\f^\alpha:\Mb\arr\Mb$ and dynamics is defined
by the map $\is{F}:\Mb \arr \Mb$ given by \Ref{Deq}. \\
\noindent (2) Equation \Ref{Deq} defines a \emph{state dependent} dynamical system. Similar
structures have been used in engineering
applications (for example, \cite{Haddad}). We indicate in 
section~\ref{Filippov_dig} a relationship with Filippov systems (this is explored further in~\cite{FA}). 
However, the notion of an integral curve for an asynchronous network is generally 
different from that of a Filippov system, see examples~\ref{path}(2). \\
(3) The network vector field does not uniquely determine $\A$, $\mathcal{E}$ or
$\mathcal{F}$. Usually, however,  the choice of $\A$, $\mathcal{E}$ and $\mathcal{F}$
is naturally determined by the problem. Sometimes it is convenient to view the
network vector field as the basic object and regard asynchronous networks
as being \emph{equivalent} if they define the same network vector field. \\
(4) Since the event sets $\{E^\aa \dd \aa \in \A\}$ partition $\Mb$, the network vector field $\is{F}$
only depends on $\f^{\aa}|E^\aa$. Rather than assume that
$\f^{\aa}$ is smooth on $\Mb$, we could have required that each $\f^{\aa}$ was defined as
smooth map in the sense of Whitney~\cite{Whitney34} on $\overline{E^\aa}$ (and so extends smoothly to $\Mb$).\\
(5) Although the vector fields $\f^\alpha \in \F$ are assumed to satisfy (N1--3), this may \emph{not} hold for
$\f^\alpha|E^\alpha$, $\alpha \in\A$. Sometimes, but \emph{not} always, there is an 
equivalent network $\Net'$ such that the dependencies of
each admissible vector field for $\Net'$ are not changed by restriction to the corresponding event set. 
\end{rems}

\subsection{Integral curves and proper asynchronous networks}
\label{reg:sec}
We start with a definition of integral curve suitable for asynchronous networks.
\begin{Def}
\label{sol:def}
Let $\Net$ be an asynchronous network with network vector field
$\is{F}$.  An \emph{integral curve} or \emph{trajectory} for $\is{F}$ with initial condition $\XX_0 \in \Mb$ is a
map $\bphi:[0,T) \arr \Mb$, $T \in (0,\infty]$, satisfying
\begin{enumerate}
\item $\bphi(0) = \XX_0$.
\item $\bphi$ is continuous.
\item There exists a closed countable subset $D$ of $[0,T)$ such that for every $u \in D$, there exists $v \in D\cup \{T\}$, $v > u$,
such that
\begin{enumerate}
\item $(u,v) \cap D = \emptyset$.
\item $\bphi$ is $C^1$ on $(u,v)$ and $\bphi'(t) = \is{F}(\bphi(t))$, $t \in (u,v)$.
\item $\lim_{t \arr u+} \bphi'(t) = \is{F}(\bphi(u))$.
\end{enumerate}
\end{enumerate}
\end{Def}
\begin{rems}
(1) It is routine to verify that if $\bpsi:[0,S)\arr \Mb$ is another integral curve with initial condition $\XX_0$, then
$\bpsi = \bphi$ on $[0,\min\{S,T\})$ (uniqueness).
As a consequence we can define the \emph{maximal}
integral curve $\bphi:[0,T_{\text{max}}) \arr \Mb$ with initial condition $\XX_0$.
In the sequel, integral curves will be maximal unless otherwise indicated.\\
(2) If $T = \infty$ in the definition, the trajectory $\bphi:\real_+ \arr \Mb$ is \emph{complete}.\\
(3) The set $D$ may have accumulation points in $D$ -- accumulation is always from the left on
account of condition (3a). In the examples we consider $D$ will always be a finite set.\\
(4) Typically, for each $u \in D$, there exists $\alpha \in \A$
such that $\Ee(\bphi(t)) = \alpha$ for $t\in(u,v)$
and so $\bphi((u,v))\subset E^\alpha$.
Condition (3c) implies that if $\Ee(\bphi(u)) = \beta \ne \alpha$, we
must have $\f^\alpha(\bphi(u)) = \f^\beta(\bphi(u))$.
\end{rems}

Without further conditions on the event map, the vector field
$\is{F}$ determined by an asynchronous network $\Net$
may not have integral curves through every point of the phase space. 
\begin{exams}
\label{path}
\begin{figure}[h]
\centering
\includegraphics[width=0.9\textwidth]{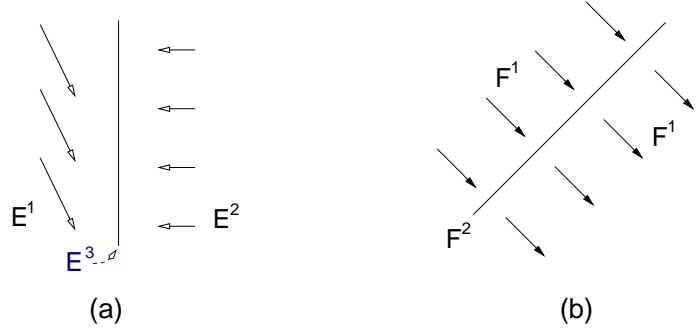}
\caption{Integral curves for the network vector field may not be well defined (a) and may differ from those given by the
Filippov conventions (b).} 
\label{ws}
\end{figure}
(1) Take event sets $E^1 = \{(x_1,x_2) \dd  x_1 \le 0\}$, $E^2 = \real^2 \sm E^1$,
and corresponding constant vector fields $\f^1 = (1,-2)$, $\f^2 = (-1,0)$ (see figure~\ref{ws}(a)). 
Trajectories cannot be continued, according to definition~\ref{sol:def}, once they meet $x_1 = 0$. 
One way round this problem is to define a new event set $E^3 = \partial E^1$ and the \emph{sliding} vector field $\f^3 = \f^1+\f^2 = (0,-2)$.
There is then a complete integral curve through every point of $\real^2$ and the corresponding semiflow $\Phi: \real^2 \times \real_+ \arr \real^2$
is continuous. This approach is based on the Filippov 
construction~\cite[Chapter 2, page 50]{Fi} where we take a vector field in the positive cone defined by
$\f^1, \f^2$ (often the unique convex combination
$\lambda \f^1 + (1-\lambda)\f^2$)
which is tangent to $\partial E^1=E^3$).

\noindent (2) Take event sets  $F^1 = \{(x_1,x_2) \dd  x_1 \ne x_2\}$, 
$F^2 = \{(x_1,x_2) \dd  x_1 = x_2\}$, and corresponding vector fields
$\f^1(x_1,x_2) = (1,-1)$,  $\f^2(x_1,x_2) = (0,0)$ (see figure~\ref{ws}(b) and note that the event $F^2$ models a collision, after which dynamics stops).
Integral curves are defined for all initial conditions in $\real^2$ but the semiflow $\Phi: \real^2 \times \real_+ \arr \real^2$
will not be continuous on $F^2$. Here the Filippov construction gives the wrong network solution -- the diagonal 
$F^2$ is regarded as a removable singularity.

We discuss the relationship between asynchronous
networks and Filippov systems further in section~\ref{Filippov_dig}; see also~\cite{FA}.
\examend
\end{exams}

\begin{Def}
\label{ammen2}
The asynchronous network $\Net$ is 
\emph{proper} if for all $\XX\in\Mb$, the maximal integral curve through $\XX$ is complete: $\bphi_\XX:[0,\infty)\arr \Mb$.
\end{Def}

\begin{rems}
\label{rem:TrajCont}
(1) If  $\Net$ is proper, 
network dynamics is given by a semiflow $\Phi:\Mb \times \real_+ \arr \Mb$.
Although $\Phi(\XX, t)$ will be continuous as a function of $t \in \real_+$, 
it need not be continuous as a function of $\XX \in \Mb$ (see examples~\ref{path}(2)).\\
(2) In many cases of interest, some of the node phase spaces $M_i$ may be open domains in $\real^n$ with with $\partial M_i \ne \emptyset$. 
Here there is the possibility that trajectories may exit $\Mb$: if 
$\bphi=(\bphi_1,\dotsc,\bphi_k)$ is a trajectory, there may exist $i \in \is{k}$ and a smallest $s > 0$ such that
$\bphi_i(s) \defoo \lim_{t \arr s_-}\bphi_i(t) \in \partial M_i$.  
The maximal domain for $\bphi$ is necessarily $[0,s)$.
Under additional hypotheses, it may be possible to extend~$\bphi$ to a complete trajectory by
setting $\is{F}_j \equiv 0$ on $\real^n \sm M_j$, $j \in \is{k}$ (the $j$th component of $\bphi$ is
stopped when it meets the boundary of $M_j$). In this way, we can regard $\Net$ as proper.  
We develop this point of view further in part II~\cite{BF1}. 
\end{rems}

Event sets are typically defined by analytic and algebraic conditions that reflect
logical conditions on the underlying dynamics. 
\begin{Def}
Let $\Net$ be an asynchronous network.
The event structure $\{E^\aa \dd \aa\in\A\}$ of $\Net$ is \emph{regular} if
the event sets $E^\aa$ are all semianalytic subsets\footnote{Defined locally by analytic equations and inequalities.
We refer to \cite{Gibson,Boc} for precise definitions and properties.} of $\Mb$. 
\end{Def}
\begin{rem}
For the examples in this paper, event sets will typically be
semialgebraic -- defined by polynomial equalities and inequalities.
\end{rem} 
\begin{Def}
\label{ammen}
An asynchronous network $\Net$ is \emph{amenable} if
\begin{enumerate}
\item The event structure $\{E^\aa \dd \aa \in \A\}$ is regular.
\item If $\XX \in E^\aa$, $\aa \in \A$, there exists
a maximal $t(\XX)\in  (0, \infty]$
such that the integral curve $\bphi_\XX$ through $\XX$ is defined on $[0,t(\XX))$ and
\[
\bphi_\XX(t) \in E^\aa,\;\; t \in [0,t(\XX)).
\]
\item Either $M_i$ is compact without boundary or $M_i = \real^{n_i}$ and
vector fields have at most linear growth on $M_i$: $\exists a,b > 0$ such that
\[
\|\is{f}_i^\aa(\XX)\| \le a+ b \|\XX\|, \;\XX \in \Mb,\; \aa \in \A.
\]
\end{enumerate}
\end{Def}

\begin{rems}
\label{defammen}
(1) Condition (2) of definition~\ref{ammen} suggests that the vector field $\is{f}^{\aa}$ should in some sense be tangent to $E^\aa$.
The issue of tangency can be made precise using the regularity assumption which implies that $E^\aa$ has a locally finite
stratification into submanifolds without boundary (for example, the canonical Whitney regular stratification of
each event set~\cite{Gibson,Mather}).  This allows us to unambiguously define tangency at points of $E^\aa$ which do not lie in the boundary
of strata. Care is needed at points lying in the boundary of strata and in the example below we indicate how the geometric
structure of the event set can impose strong constraints on associated vector fields.\\
(2) If an event set is a closed submanifold without boundary, it follows from definition~\ref{ammen}(2) that
any trajectory that meets the event set will never leave the event set.\\
(3) In part II we extend definition~\ref{ammen}(3) to allow for trajectories to exit the domain and stop (see remark~\ref{rem:TrajCont}(2)).\\ 
(4) We may extend the definition of amenability to include asynchronous networks which are equivalent to
an amenable network.
\end{rems}

\begin{exams}
\label{amex}
Take $k = 2$, $M_1 = M_2 = \real$. \\
(1) As event sets take the semialgebraic subsets of $\real^2$ defined by
\[
E^1 = \{(x,0) \dd x < 0\},\;E^2 = \{(0,y) \dd y > 0\},\; E^0 = \real^2\sm \bigcup_{i=1}^2 E^i.
\]
The event sets are neither open nor closed. We define associated vector fields $\f^j$, $j \in \iz{2}$, on $\real^2$ by
\[
f^1(x,y) = (1,0),\; f^2(x,y) = (0,-1),\; f^0 = f^1+f^2.
\]
It is a simple exercise to verify that the network is amenable and proper but that the associated semiflow
$\Phi: \real^2 \times \real_+ \arr \real^2$ is not continuous along $E^1$ or $E^2$ (it is continuous at $(0,0)$). 

\noindent (2) Suppose that the event set $E^1$ is the cusp defined by $\{(x,y) \in \real^2 \dd x \ne 0,\;y^2 = x^3\}$ and
$E^2 = \real^2 \sm E^1$.
In this case any smooth ($C^1$ suffices) vector field on $\real^2$
which is tangent to $E^1$ must vanish at $\{(0,0)\}$ (an example of such a vector field is $(2ax,3ay)$, $a \in \real$). If we require amenability,
then all trajectories which meet $E^1$ will never leave $E^1$.
\examend
\end{exams}

\begin{prop}
\label{reglemma1}
An amenable asynchronous network is proper.
\end{prop}
\begin{proof}We give details for the case when $\Mb$ is compact. Fix $\XX \in \Mb$. Suppose that $\bphi_i: [0,s_i) \arr \Mb$ are forward
trajectories for $\is{F}$ through $\XX$, $i \in \is{2}$. Using uniqueness of solutions of
differential equations and definition~\ref{ammen}(2), it is easy to see that $\bphi_1 = \bphi_2$ on $[0,s_1) \cap [0,s_2)$.
It follows that if we define
\[
T = \sup\set{t}{\text{there is a trajectory $\bpsi:[0,t)\arr \Mb$ through $\XX$}}
\]
then we have a unique trajectory $\bphi: [0,T) \arr \Mb$ through $\XX$. If $T = \infty$, we are done. But if $T < \infty$, then
we can extend $\bphi$ to $[0,T]$ by $\bphi(T) = \lim_{t \arr T-} \bphi(t)$ (remarks~\ref{defammen}(3)). If $\bphi(T) \in E^\aa$ then by
definition~\ref{ammen}(2), $\bphi$ extends to $[0,T+t(\bphi(T)))$, where $t(\bphi(T)) > 0$. This contradicts the maximality of $T$ and so
$T = \infty$.
\end{proof}
\begin{rems}
(1) Proposition~\ref{reglemma1} says nothing about the number of changes in 
the event map that occur along a trajectory. Without further conditions, 
there may be a countable infinity of changes with countably many accumulation points (see definition~\ref{sol:def}
and note the analogy with Zeno-like behaviour~\cite{BBCK}).\\
(2) As shown in examples~\ref{amex}(1), the semiflow given by proposition~\ref{reglemma1} need not be continuous (as
a function of $(\XX,t)$). \\
(3) Amenability is sufficient but not necessary for properness.
\end{rems}

\subsection{Semiflows for amenable asynchronous networks}
Assume $\Net$ is an amenable asynchronous network with network vector field $\is{F}$.
For each $\aa \in \A$, denote the flow of $\is{f}^\aa$ by $\Phi^\aa$.

Let $\XX \in \Mb$ and $\bphi:\real_+ \arr \Mb$ be the maximal integral curve through $\XX$ for $\is{F}$. If follows from
the definition of integral curve and amenability that there is a countable closed subset $D = D(\XX)$ of $\real_+ \cup \{\infty\}$ such that for each
$u \in D$, there exist unique $\alpha \in \A$, $v = v(u) \in D$ such that 
\[
(u,v) \cap D = \emptyset, \; \Ee(v) \ne \alpha, \; \bphi([u,v)) \subset E^\alpha.
\]
(For $\Ee(u) = \alpha$ we need amenability.)
\begin{prop}
Let $\Net$ be an amenable asynchronous network. Suppose that for all $\XX \in \Mb $,  $D(\XX)$ is finite
and set $D(\XX) = \{ t^\XX_j \dd 0 = t^\XX_0 < t^\XX_1 < \dotsc < t^\XX_N < t^\infty_{N+1} = \infty\}$, $\alpha^\XX_j = \Ee(\bphi(t^\XX_j))$, $j \in \iz{N}$.
The semiflow $\Phi: \Mb \times \real_+ \arr \Mb$ for $\is{F}$ is given in terms of the flows $\Phi^\aa$ by
\[
\Phi_\XX(t) = \Phi^{\aa_p^\XX}(\cdots \Phi^{\alpha_1^\XX}(\Phi^{\aa^\XX_0}(\XX,t^\XX_1),t^\XX_2-t^\XX_1)\cdots, t -t^\XX_p),
\]
where $t \in [t^\XX_p,t^\XX_{p+1})$,  $p \in \iz{N}$. 
\end{prop}
\proof For $t \in  [t^\XX_p,t^\XX_{p+1})$, $\Phi^{\aa_p^\XX}(\XX_p,t)$ is the solution to $\XX'(t) = \f^{\alpha_p^\XX}(\XX)$ 
with initial condition $\XX_p = \Phi_\XX(t_p^\XX)$.  \qed

\subsection{Asynchronous networks with additive input structure}
\label{sec:AddStruct}

A natural source of asynchronous networks comes from synchronous networks with additive input structure.
The event map can be either state dependent (with constraints) or stochastic (see the following section).

Fix a~$k$ node synchronous 
network $\Nn$ with additive input structure and network vector field $\f=(f_1,\dotsc,f_k)$ given by. 
\begin{equation} 
\label{EQais2}
f_i(\is{x}_i;\is{x}_{j_1},\dotsc, \is{x}_{j_{e_i}}) = F_i(\is{x}_i) + \sum_{s=1}^{e_i} F_{ij_s}(\is{x}_{j_s} , \is{x}_i), 
\;i \in \is{k}.
\end{equation}
On account of the additive input structure, it is natural to remove and later reinsert connections 
between nodes. 

For $i \in \is{k}$, let $(W_i,\mc{L}_i)$ be the constraint defined by the $0$-dimensional foliation of $W_i = M_i$. 
If dynamics on $N_i$ is constrained, then dynamics is stopped: $\xx_i' = 0$. 
Let~$\Gamma$ be the network graph determined by~\Ref{EQais2} with 
associated \ZO~matrix $\gamma \in M(k)$. Take $\is{P} = (1,\dotsc,1)$ and let $\A \subset \CC{k}$ be 
a generalized connection structure such that
\begin{enumerate}
\item $(0 \dd \gamma) \in \A$,
\item for all $\aa = (\az\dd \af)$ the matrix 
$\af$ defines a subgraph of $\Gamma$, and
\item $\aa_{i0} \in \{0,1\}$ for all $i \in \is{k}$, $\aa \in \A$.
\end{enumerate}
For each $\aa\in\A$, define the $\aa$-admissible vector field 
$\is{f}^\aa$ by
\begin{equation*} 
f^\aa_i(\is{x}_i;\is{x}_{j_1},\dotsc, \is{x}_{j_{e_i}}) = (1-\aa_{i0})\left(F_i(\is{x}_i) + \sum_{s=1}^{e_i} \aa_{ij_s}F_{ij_s}(\is{x}_{j_s} , \is{x}_i)\right),
\;i \in \is{k},
\end{equation*}
and set $\F = \{\f^\alpha \dd \alpha \in \A\}$.
If we choose an event map $\Ee: \Mb\to\A$ and 
take $\F = \set{\is{f}^\aa}{\aa\in \A}$, then 
$\Net = (\Nn, \A, \F, \Ee)$ is an asynchronous network. 
We refer to $\Net$ as an \emph{asynchronous 
network with additive input structure}.

For $\alpha \in \A$, $i \in \is{k}$, let $J(i, \aa) = \{j \dd \alpha_{ij} =1, j \in \iz{k}\}$ be the
dependency set of $f_i^\alpha$.

\begin{Def}
An asynchronous network $\Net$ is 
\emph{input consistent} if for any node~$N_i$ and $\aa, \beta\in\A$ 
with dependency sets satisfying $J(i, \aa) = J(i, \beta)$
we have 
$f_i^\aa = f_i^\beta$.
\end{Def}

As an immediate consequence of our constructions we have
\begin{lemma}
Asynchronous networks with additive input structure are input
consistent.
\end{lemma}

In summary, if $\Net$ is an asynchronous network with additive input structure all the admissible vector fields
are derived from the network vector field of a synchronous network. 
\subsection{Local clocks on an asynchronous network}

In this section we describe \emph{local clocks} on an asynchronous network. We give only brief details sufficient for the examples we give later 
(the general set up appears in~\cite{BF2}). Roughly speaking, a local clock will be associated to a set of nodes, or connections, and may be thought of
thought of as a stopwatch with time $\tau \in \real_+$. In particular, the local clock will run intermittently and switching between 
on and off states will be determined by thresholds. 

Fix  a finite set of nodes $\Nn= \{N_0,N_1,\dotsc,N_k\}$ with associated phase spaces $M_i$, $i \in \is{k}$,
a generalized connection structure $\A \subset  \CC{k}$ and a constraint structure $\is{C}$. Local clocks will be
defined in terms of strongly connected components of elements of $\A$.

Suppose that $\alpha \in \A$ and let $\beta,\gamma$ be distinct strongly connected components of $\alpha$ with
respective node sets $A \subset \is{k}$, $B \subset \iz{k}$. A local time $\tau_{\beta,\gamma} \in \real_+$ will be defined on $\beta$ (or the nodes $A$)
if there exists a connection $N_j \arr N_i$, $j \in B$, $i \in A$.
\begin{exams}
(1) The constraining node $\NS$ is always a strongly connected component of $\alpha$. If $\alpha=\NS\arr N_i$, then
we may take $\beta = \{N_i\}$, $\gamma = \{\NS\}$ and define the local time $\tau_i$ on $N_i$.\\
(2) If $\alpha = \NS \arr N_i \leftrightarrow N_j \leftarrow \NS$, then we may take $\beta = N_i \leftrightarrow N_j$, $\gamma = \{\NS\}$ and
obtain the local time $\tau_{\beta} = \tau_{ij}$ defined on $N_i, N_j$ (or $N_i \leftrightarrow N_j$).
\examend
\end{exams}

Choose a set $\tau_1,\dotsc,\tau_s$ of local times and set 
\[
\mc{T} = \real_+^s = \{\boldsymbol{\tau}=(\tau_1,\dotsc,\tau_s) \dd \tau_1,\dotsc,\tau_s \in \real_+\}.
\]
We extend the phase space of $\Nn$ to $\mc{M} = \Mb \times \mc{T}$. Given $\alpha \in \A$, an $\alpha$-admissible vector field
$\f^\alpha$ on $\mc{M}$ will be a smooth vector field of the form
\[
\f^\alpha(\XX,\boldsymbol{\tau}) = (f_1^\alpha(\XX,\boldsymbol{\tau}),\dotsc,f_k^\alpha(\XX,\boldsymbol{\tau}),h_1,\dotsc,h_s),
\]
where $h_1,\dotsc,h_s \in \{0,1\}$ are constant vector fields. 

Just as before, we define an $\A$-structure $\F$, an event map $\Ee: \mc{M} \arr \A$ and associated asynchronous network $(\Nn,\A,\F,\Ee)$. Our previous definitions and results 
continue to apply. 
\begin{exam}
Suppose $k = 1$, $\Nn = \{N_0,N_1\}$, and $M_1 = \real$. Choose a smooth vector field $f:\real\arr\real$ such that $1 \ge f(x) > 0$ for all $x \in \real$. 
Define $\A = \{\ECS, \alpha = \NS \arr N_1\}$. Define the local time $\tau \in \real_+$ associated to $\alpha$. Set $\mc{M} = \real \times \real_+$. 
Define $\F = \{\f^\ECS,\f^\aa\}$ by
\[
\f^{\ECS}(x,\tau) = (f(x),0),\; \f^{\aa}(x,\tau) = (0,1),\; (x,\tau) \in \mc{M}.
\]
Fix $T > 0$ and define the event map $\Ee: \mc{M} \arr \A$ by
\begin{eqnarray*}
\Ee(x,\tau) & = & \ECS, \;\text{if } x \ne 0\; \text{or } \tau \ge T\\
& = & \alpha, \;\text{if } x = 0\; \text{and } \tau < T
\end{eqnarray*}
The asynchronous network $(\Nn,\A,\F,\Ee)$ is amenable.
If we initialize at $(x_0,0)$, $x_0 < 0$, then the system evolves until $x = 0$,  stops for local time $T$ seconds and then restarts. 
In practice, the local clock is reset to zero after the system restarts.
\examend
\end{exam}

\subsection{Stochastic event processes and asynchronous networks}
Given node set $\Nn$, constraint structure $\is{C}$, generalized connection structure $\A$ and
$\A$-structure $\F$, an
\emph{event process} is a state dependent stochastic process $\Ee_{(t,\XX)}$ taking values in
$\A$.

\begin{Def}
\label{defasync_stoc0}
(Notation as above.)
A \emph{stochastic asynchronous network} $\Net$ is a quadruple
$(\mathcal{N},\mathcal{A},\mathcal{F},\Ee)$, where
$\Ee=\Ee_{(t,\XX)}$ is an event process.
\end{Def}

In the most general case there are no restrictions on the process
$\Ee_{(t,\XX)}$: there may be (stochastic) dependence on time~$t\in\R^+$, pure space
dependence ($\Ee_{(t,\XX)} = \Ee(\XX)$), or both. If $\Ee_{(t,\XX)}$ is independent of time, then
the event process reduces to an event map $\Ee: \Mb\arr \A$. If $\Ee_{(t,\XX)}$ is independent of $\XX$,
then under mild conditions on $\Ee$, such as assuming $\Ee$ is Poisson, integral curves 
on the stochastic asynchronous network
$(\Nn,\A,\F,\Ee_t)$ will be almost surely piecewise smooth.  

We discuss stochastic asynchronous networks in more detail in~\cite{BF2}. We give one simple example here
related to additive input structure.
\begin{exam}
We follow the assumptions and notational conventions of section~\ref{sec:AddStruct} and assume given a synchronous network with additive
input structure and dynamics given by~\Ref{EQais2}. Let $\A$ be a generalized connection structure and $\Ee$ be a time dependent
event process taking values in $\A$. Assume $\Mb$ is compact and the set of times $t_0 < t_1 < \dotsc $ where the 
connection structure changes has Poisson statistics. The stochastic asynchronous network $(\Nn,\A,\F,\Ee)$ is an example of 
a stochastic asynchronous networks with additive input structure. Almost surely, trajectories will be piecewise smooth
and defined for all positive time. 
\examend
\end{exam}

%% file: sectionI_05.tex
\section{Model examples of asynchronous networks} 
\label{ampex}
In this section, we describe two asynchronous networks using the formalism and ideas developed in the previous section.
We refer also to~\cite{BF2}, for the detailed description of an asynchronous network modelling spiking neurons, adaptivity and learning (STDP). 

\renewcommand{\star}{*}

\subsection{A transport example: train dynamics}
\label{basic}
We use a simple transport example -- a single track line with a passing loop -- to illustrate characteristic features
of asynchronous networks in a setting requiring minimal structure and background knowledge.

Consider two trains $\mathfrak{T}_1, \mathfrak{T}_2$ travelling in opposite directions along a single track
railway line; see figure~\ref{strl}. We assume no central control and no communication between train drivers 
unless both trains are in the passing loop.

\begin{figure}[h]
\centering
\includegraphics[width=\textwidth]{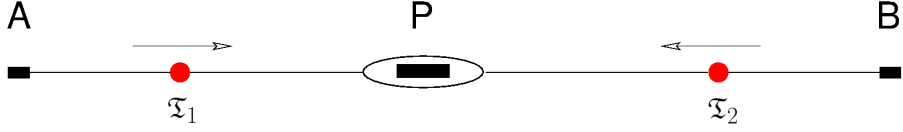}
\caption{Two trains on a single track railway line with a passing loop and stations.}
\label{strl}
\end{figure}

Take as phase spaces for the trains the closed interval $I = [-a,b]$, where $a,b > 0$. Suppose the end points of $I$ correspond to the
stations $A$ (at $-a$) and $B$ (at $b$) and that the passing loop is at $0 \in I$. Assume that the passing loop
is associated with a third station $P$. 

The position of train $\mathfrak{T}_i$ at time $t\ge 0$ will be denoted by $x_i(t) \in I$, $i \in \is{2}$. Suppose
that $x_1(0) = -a$, $x_2(0) = b$. Assume that, outside of the stations $A,B,P$,
the velocity of the trains is given by smooth vector fields $V_1, V_2:I\arr\real$ satisfying
\[
V_1(x) > 0 > V_2(x),\; x \in I.
\]
That is, $\mathfrak{T}_1$ is moving to the right and $\mathfrak{T}_2$ to the left. 
In order to pass each other, the trains must enter the passing loop and stop at $P$. 

Fix thresholds $S,S_1,S_2,T_1,T_2 \in \real_+$. Train $\mathfrak{T}_i$ will depart at time $T_i$, $i \in \is{2}$.
We require that trains have to be together in station $P$ for time $S$ and,  additionally, the train $\mathfrak{T}_i$ 
must be in the station for time $S_i$, $i \in \is{2}$ (this is an additional condition on $\mathfrak{T}_i$ only if $S_i > S$).  
The trains can move out of the station when these thresholds are met.
Note that 
the trains will not generally leave the station at the same time if $S_1>  S$ or $S_2 > S$. We model train dynamics by an asynchronous network.

First we discuss connection structures. Associate the node $N_i$ with train $\mathfrak{T}_i$, $i \in \is{2}$.
Train $\mathfrak{T}_i$ will be  stopped at $P$ only if there is a
connection $\aa_i = \NS \arr N_i$, $i \in \is{2}$. We only allow communication between trains when both trains are
stopped at $P$. In this case, the connection structure will be $\beta =N_0 \arr N_1 \leftrightarrow N_2 \leftarrow N_2$.
If either train is not stopped at $P$, there is no connection between the trains.

As the drivers of the trains cannot communicate (unless both trains are in the station $P$) 
and there is no central control, the times associated 
with the thresholds $S_1,S_2$ will be local times. 
Specifically, when train $\mathfrak{T}_i$ stops at $P$, the driver's stopwatch will be started. 
This will be a local time $\tau_i$ for $\mathfrak{T}_i$ and associated to the connection $\NS \arr N_i$,  
When both trains are stopped at $P$, we use a third local time $\tau = \tau_{12}$ associated 
to the connection $N_1 \leftrightarrow N_2$ (alternatively, the drivers could synchronize their stopwatches 
but still the stopwatches may not run at the same speed). 

We describe this setup using our formalism for asynchronous networks.
As network phase space we take 
\[
\mc{M} = \{(\XX,\tb) = (x_1,x_2,\tau_1,\tau_2,\tau) \dd x_1,x_2 \in I,\; \tau_1,\tau_2,\tau \in \real_+\} = I^2 \times\real_+^3.
\]
We define the generalized connection structure $\A = \{ \aa_1, \aa_2,\beta,\ECS\}$ and
let $\F$ be the $\A$-structure given by 
\begin{eqnarray*}
\f^\ECS(\XX,\tb)&=&((V_1(x_1), V_2(x_2)),(0,0,0))\\
\f^{\aa_1}(\XX,\tb)&=&((0, V_2(x_2)),(1,0,0))\\
\f^{\aa_2}(\XX,\tb)&=&((V_1(x_1), 0),(0,1,0))\\
\f^{\beta}(\XX,\tb)&=&((0, 0),(1,1,1))
\end{eqnarray*}
We define the event map $\mathcal{E}: \mc{M}\arr \A$ by
\[ {\small
\mathcal{E}(\XX,\tb)  =  
\begin{cases}
\alpha_1& \text{if }(x_1 = 0, x_2 > 0) \vee ( (x_1 = 0, x_2 \le 0) \wedge (\tau_1 < S_1)) \\
\alpha_2& \text{if }(x_2 = 0, x_1 < 0) \vee ( (x_2 = 0, x_1 \ge 0) \wedge (\tau_2 < S_2)) \\
\beta&\text{if } (x_1 = x_2 = 0) \wedge ((\tau < S) \vee ((\tau_1 < S_1) \wedge (\tau_2 < S_2)))\\
\ECS& \text{otherwise}.
\end{cases}
}\normalsize 
\] 
\noindent Here we have used the logical connectives $\vee$ for \emph{or} and $\wedge$ for \emph{and}.
Dynamics on the asynchronous network $\Net=(\Nn,\A,\F,\Ee)$ is 
given by the vector field $\is{F}(\XX) = \f^{\Ee(\XX)}(\XX)$.
Provided that we initialize so that $x_1(0) < 0 < x_2(0)$, $\tau_1(0) = \tau_2(0) = \tau(0)=0$, it is easy to see that
$\Net$ is amenable.  
\subsubsection{Initialization, termination and function}
\label{sec:FuncExample}
The network $\Net$ has a function: each train has to traverse the line to reach the opposite station. Thus we 
can regard $\Net$ as a \emph{functional asynchronous network}. Formally, 
define \emph{initialization}  and \emph{termination} sets by $\mathbb{I}_1 = \{-a\}$, $\mathbb{I}_2 = \{b\}$
and  $\mathbb{F}_1 = \{b\}$, $\mathbb{F}_2 = \{-a\}$ respectively.
We call $\mathbb{I} = \mathbb{I}_1 \times \mathbb{I}_2$ and $\mathbb{F} = \mathbb{F}_1 \times \mathbb{F}_2$  the
initialization and termination sets for $\Net$.  The function of the network is to get from $\mathbb{I}$ to $\mathbb{F}$ in finite time.

Typically, the thresholds $S,S_1,S_2,T_1,T_2 \in \real_+$ will be chosen stochastically. For example,
the starting times $T_1, T_2$ according to an exponential distribution. If we initialize at $(-a,T_1), (b,T_2)$, and
take $\tau_1(0) = \tau_2(0) = \tau(0) = 0$, it is
easy to verify that solutions will be defined and continuous for all positive time under 
the assumption that a train stops  when it reaches its termination set.
\subsubsection{Adding dynamics}
\label{dynamics}
The trains only ``interact'' when both are stopped at $P$.
We now add a non-trivial dynamic interaction 
when the trains are stopped at $P$. To this end, we
additionally require that
\begin{enumerate}
\item The drivers are running oscillators of approximately the same frequency
(randomly initialized at the start of the trip).
\item When both trains are at $P$, the oscillators are cross-coupled allowing for eventual approximate frequency synchronization.
\item The trains cannot  restart until the oscillators have phase synchronized to within $\varepsilon$, where $0 < \varepsilon < 0.5$.
\end{enumerate}
For example, fix $\omega_1, \omega_2 \in \real$ and define $H(\theta) = k\sin 2\pi \theta$, $\theta \in \mathbb{T}$, where $k > 0$.
Take as network phase space $\mc{M}^\star = \mc{M} \times \mathbb{T}^2$. Define vector fields $\is{h}^\ECS = \is{h}^{\aa_1} = \is{h}^{\aa_2}$
and $\is{h}^{\beta}$ on $\mc{M}^\star$ by
\begin{eqnarray*}
\is{h}^\ECS(\XX,\tb,\theta_1,\theta_2)& =& (\is{0},\is{0},\omega_1,\omega_2)\\ 
\is{h}^\beta(\XX,\tb,\theta_1,\theta_2)& =& (\is{0},\is{0},\omega_1+ H(\theta_2-\theta_1), \omega_2 +  H(\theta_1-\theta_2))
\end{eqnarray*}
Define a new $\A$-structure $\mc{F}^\star$ by
\[
\is{g}^\ECS = \f^\ECS+\is{h}^\ECS,\; \is{g}^{\aa_1} = \f^{\aa_1}+\is{h}^{\aa_1},\; \is{g}^{\aa_2} = \f^{\aa_2}+\is{h}^{\aa_2},\;
\is{g}^\beta =  \f^\beta+ \is{h}^\beta,
\]
where $\f^\ECS, \is{f}^{\aa_1}, \is{f}^{\aa_2}, \is{f}^\beta\in\F$ do not depend on $(\theta_1,\theta_2) \in \mathbb{T}^2$.
Modify the event map $\Ee$ by requiring that $\Ee(\XX,\tb,\theta_1,\theta_2) = \beta$ iff
\[
(x_1 = x_2 = 0) \wedge ((\tau < S)\vee (|\theta_1-\theta_2| > \varepsilon) \vee ((\tau_1 < S_1) \wedge (\tau_2 < S_2)))
\]
In this case, for almost all initializations, the oscillators will eventually phase synchronize to 
within $\varepsilon$ provided that $\sin^{-1}(|\omega_1 - \omega_2|/2k) < 2 \pi \varepsilon$. In particular, if $\omega_1 = \omega_2$, the oscillators
will synchronize unless $|\theta_1(0)-\theta_2(0)|=0.5$. 

\subsubsection{Relations with Filippov systems}
\label{Filippov_dig}
Assume all the thresholds of our model are zero.
Note that if  $S = S_1 = S_2 = 0$, then there is no need for local clocks and we may model by the 
asynchronous network $\Net^\star = (\Nn,\A^\star,\F^\star,\Ee^\star)$, where
$\A^\star = \{\aa_1,\aa_2,\ECS\}$, $\F^\star = \{\f^\ECS,\f^{\aa_1},\f^{\aa_2}\}$, where
$\f^\ECS(\XX)=(V_1(x_1), V_2(x_2))$, $\f^{\aa_1}(\XX)=(0, V_2(x_2))$,
$\f^{\aa_2}(\XX)=(V_1(x_1), 0)$,
and the event map $\Ee^\star$ is defined by
\[
\mathcal{E}^\star(\XX)  =  
\begin{cases}
\alpha_1& \text{if }x_1 = 0, x_2 > 0 \\
\alpha_2& \text{if }x_2 = 0, x_1 < 0\\
\ECS& \text{otherwise}.
\end{cases}
\]
We show dynamics for $\Net^\star$ in figure~\ref{trains1} under the initialization assumption that $x_1(0) \le  0 \le x_2(0)$.
\begin{figure}[h]
\centering
\includegraphics[width=0.75\textwidth]{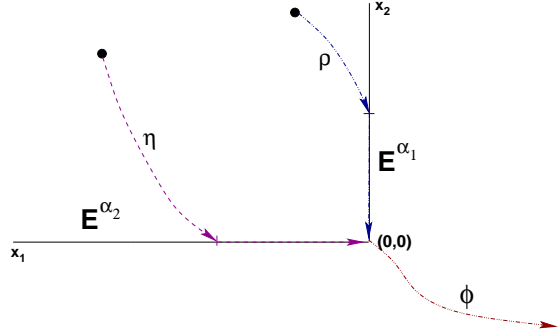}
\caption{Dynamics on a one track line with passing loop.}
\label{trains1}
\end{figure}
Referring to the figure, trajectory $\eta$ corresponds to train $\mathfrak{T}_2$ reaching $P$ first
and restarting only when $\mathfrak{T}_1$ reaches $P$. Train $\mathfrak{T}_1$ reaches $P$ first for the
trajectory $\nu$. Regardless of which train reaches $P$ first, the `exit trajectory'
$\phi$ is always the same and so there is a reduction to 1-dimensional dynamics. 
If both trains arrive simultaneously at $P$, neither stops.

\begin{figure}
\centering
\includegraphics[width=0.8\textwidth]{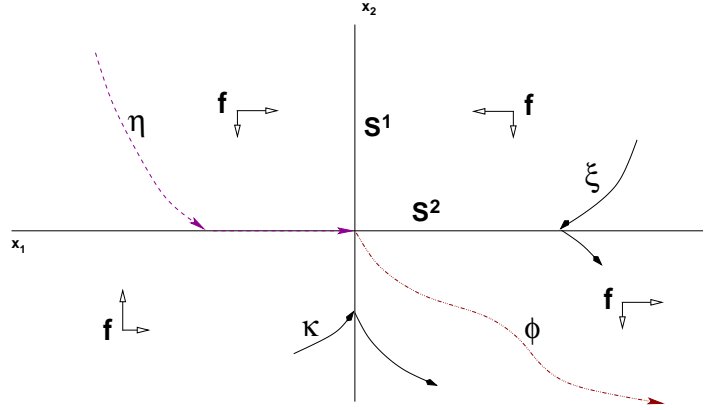}
\caption{Dynamics for the Filippov system. Trajectories $\eta$ and $\phi$ are unchanged; trajectories $\kappa$ and $\xi$
correspond to one train reversing after the other train enters the passing loop and are artifacts of the Filippov representation.}
\label{fil}
\end{figure}
The dynamics shown in figure~\ref{trains1} is suggestive of a Filippov system~\cite{Fi,BBCK} and it
is natural to ask whether there are connections between asynchronous network and 
Filippov systems. Set $\real^2_\circ = \{(x_1,x_2) \dd x_1x_2 \le 0\}$ and observe that dynamics on $\Net^\star$ is
given by a continuous semiflow $\Phi^\star:\real^2_\circ \times \real_+\arr\real_\circ^2$. 
We define a Filippov system on $\real^2$, with continuous semiflow $\Phi:\real^2 \times \real_+ \arr \real^2$, such that
$\Phi = \Phi^\star$ on $\real^2_\circ$. To this end we let $Q_{ij}$, $i,j \in \{+,-\}$ denote the closed quadrants of
$\real^2$ (so $Q_{+-} = \{(x_1,x_2) \dd x_1 \ge 0, x_2 \le 0\}$, etc) and define smooth vector fields
on each quadrant by 
\begin{eqnarray*}
\is{V}_{++}(x_1,x_2) & = & (-V(x_1),V_2(x_2)),\; (x_1, x_2) \in Q_{++} \\
\is{V}_{+-}(x_1,x_2) & = & (V(x_1),V_2(x_2)),\; (x_1, x_2) \in Q_{+-}\\
\is{V}_{--}(x_1,x_2) & = & (V(x_1),-V_2(x_2)),\; (x_1, x_2) \in Q_{--}\\
\is{V}_{-+}(x_1,x_2) & = & (V(x_1),V_2(x_2)),\; (x_1, x_2) \in Q_{-+}.
\end{eqnarray*}
These vector fields uniquely define a smooth vector field $\is{V}$ on the union of the interiors of
the quadrants. We extend $\is{V}$ to a piecewise smooth vector field on $\real^2 \sm \{(0,0)\}$ using 
the Filippov conventions. Thus, we regard the $x_i$-axis as a sliding line $S^i$, $i \in \is{2}$,
and define $\is{V}$ on $\partial Q_{-+} \cap \partial Q_{--}=E^{\aa_2}\subset S^1$ to be the unique convex combination
of $\is{V}_{-+}$ and $\is{V}_{--}$ which is tangent to $S^1$ (in this case $(\is{V}_{-+}+\is{V}_{--})/2$).
Finally define $\is{V}(0,0) = (V_1(0),V_2(0))$.
The piecewise smooth vector field $\is{V}$ has a continuous flow $\Phi:\real^2 \times \real_+\arr \real^2$ 
(integral curves are defined using the standard conventions of piecewise smooth dynamics -- see~\cite{Fi})
and $\Phi|\real_\circ^2 = \Phi^\star$.
Of course, the semiflow on
$\real^2\sm\real_\circ^2$ does not have an interpretation
in terms of trains on a line with a passing loop (see figure~\ref{fil}).

In an asynchronous network, dynamics on event sets
is given explicitly rather than by the conventions used in
Filippov systems.  However, as we have shown,
asynchronous networks can sometimes be locally represented by a Filippov system
(see~\cite{FA} for more details and greater generality). 
This relationship suggests the possibility of
applying methods and results from the extensive bifurcation theory of nonsmooth systems 
to asynchronous networks.
\subsubsection{Combining and splitting nodes}
We conclude our discussion of asynchronous networks modelling transport with a brief 
description of processes defined by combining or splitting nodes (a dynamical version of a 
\emph{Petri Net}~\cite{PNet}). We consider the simplest cases of two trains combining to form a 
single train or one train splitting to form two trains. We only give details for the first case but note that
both situations are easily generalized and also, like much of what we have discussed above,
apply naturally to production networks.

Consider node sets $\Nn^a = \{\NS,N_1,N_2\}$ and $\Nn^b = \{\NS,N_{12}\}$, where $N_1,N_2,N_{12}$ have phase space $\real$
and correspond to trains $\mathfrak{T}_1, \mathfrak{T}_2$, $\mathfrak{T}_{12}$ respectively.   We give a network formulation of
the event where trains $\mathfrak{T}_1, \mathfrak{T}_2$ are combined to form a single train $\mathfrak{T}_{12}$ (see figure~\ref{comb}).
\begin{figure}[h]
\centering
\includegraphics[width=0.75\textwidth]{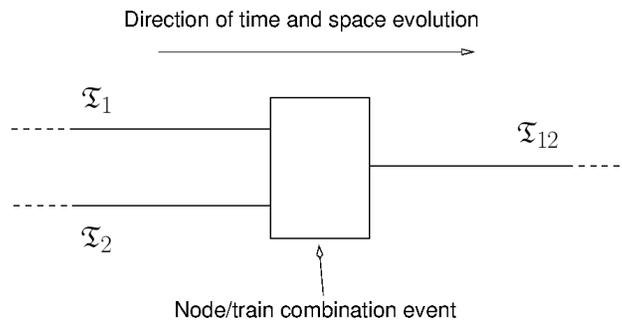}
\caption{Combining two trains into a single train.}
\label{comb}
\end{figure}
Fix vector fields $V_1, V_2, V_{12}$ on $\real$ and assume $V_1(x), V_2(x), V_{12}(x) > 0$ all $x \in \real$.
Define generalized connections structures 
\begin{eqnarray*}
\A^a&=&\{\ECS, \alpha_1 = \NS \arr N_1, \alpha_2 = \NS \arr N_2, \beta = \NS\arr N_1 \leftrightarrow N_2 \leftarrow N_2\}, \\
\A^b&=&\{ \ECS, \gamma=\NS \arr N_{12}\}.
\end{eqnarray*}
Assume a local clock with time $\tau = \tau_{12}$ that is shared between the connection $\beta \in \A^a$ and $\gamma \in \A^b$.
Define network phase spaces for $\Nn^a$, $\Nn^b$ to be $\mc{M}^a=\real^2 \times \real_+$, $\mc{M}^b=\real\times \real_+$ respectively.  
Define the $\A^a$-structure $\F^a$ by
\[
\f^{a,\ECS} = ((V_1,V_2),0),\; \f^{a,\aa_1} = ((0,V_2),0),\; \f^{a,\aa_2} = (V_1,0),\;\f^{a,\beta} = ((0,0),1).
\] 
and the $\A^b$-structure $\F^b$ by $\f^{b,\ECS} = (V_{12},0),\; \f^{b,{\gamma}} = (0,1)$.

Fix thresholds $S_2, S_1 > 0$. The threshold $S_1$ gives the time taken to combine $\mathfrak{T}_1$ and $\mathfrak{T}_2$,
and $S_2$ models the time $\mathfrak{T}_{12}$ spends in the station before leaving. Initialize $\Nn^a$ so that $x_1(0), x_2(0) < 0$ and
$\tau(0) = 0$.  The event map $\Ee^a(\XX,\tau)$ is defined for $x_1, x_2 \le 0$ and $\tau < S$ by
\begin{eqnarray*}
\Ee^a(\XX,\tau)&=&\ECS, \; x_1,x_2 < 0\\
& = & \aa_1, \; x_1 = 0, x_2 < 0 \\
& = & \aa_2, \; x_1 < 0, x_2 = 0\\
& = & \beta ,\; x_1 =x_2 = 0, \tau < S_1
\end{eqnarray*}
The event map $\Ee^b(x_{12},\tau)$ is defined for $x_{12} \ge 0$ and $\tau \ge S_1$ by
\begin{eqnarray*}
\Ee^b(x_{12},\tau)&=&\gamma, \; x_{12}= 0, \tau < S_1 + S_2\\
& = & \ECS,\; \text{otherwise}
\end{eqnarray*}
When $\tau = S_1$, we switch from network $\Nn^a$ to $\Nn^b$.

The splitting construction is similar except that we need to split the local clock for the combined train into two clocks, one for each separated train.

\subsection{Power grids and microgrids}

\subsubsection{Power grids as asynchronous networks} 
We first consider an unrealistic, but simple and instructive model that shows how asynchronous and event dependent effects can naturally fit into the framework of power grids. 
In the following section, we describe how more  realistic models are obtained, their limitations, and where we might expect asynchronous network models to be useful.

We use the simplest model~\cite{FNP} for power grid frequency stability that assumes generators are synchronous, loads are synchronous motors and consider the network of mechanical phase oscillators
\begin{equation}
\label{pn}
\theta_j'' + \alpha_j \theta_j' = P_j - \sum_{i=1}^n k_{ij}\sin (\theta_i - \theta_j), \; j \in \is{n},
\end{equation}
where~$(k_{ij})$ is a a symmetric matrix, all entries positive (zero is
allowed). If $\sum_j P_j = 0$ the system can reach an equilibrium ($P_j < 0$ corresponds to a load).
Let $\Gamma$ be the (undirected) graph determined by the matrix of connections given by $(k_{ij})$.
While the network described by~\Ref{pn} is not asynchronous (and the main
interest lies with the stability of the equilibrium solution), the dynamics
of real-world power grids are subject to factors that cannot
be adequately described by a synchronous model. For integrity of transmission lines, as well as system stability, it is essential that the phase
differences $|\theta_i - \theta_j|$ are bounded away from $\pi/2$. For example, we might require $|\theta_i - \theta_j|\le T_{ij}$, 
where $T_{ij} \in (0,\pi/2)$ will be a threshold determining the safe operational load for the transmission line. This leads to the 
construction of state dependent event maps $\Ee_{ij}: \mathbb{T}^n\arr \{\Gamma,\Gamma\sm \{i\leftrightarrow j\}\}$. If $|\theta_i - \theta_j|> T_{ij}$,
then $\Ee_{ij}(\boldsymbol{\theta}) = \Gamma\sm \{i\leftrightarrow j\}$ and the transmission line between nodes $i$ and $j$ is disconnected.
Equation \Ref{pn} is modified accordingly.  
Similarly, lines or generators may be disconnected because of external events -- such
as lightening strikes or mechanical breakdowns. These can be modelled using a stochastic event map. 

As indicated above, this model is unrealistic
(it is not true, for example, that typical loads are synchronous motors). In the next section, we indicate how more realistic models
are obtained, their limitations, and where we might expect asychronous network models to be useful.

\subsubsection{Network-reduced model for power grids}
We give an overview of the network-reduced coupled phase oscillator model for power grids, largely based on D\"orfler~\cite{DT},
and refer the reader to \cite{Motter2015,DT} for greater generality, alternative models, and the many details we omit. Apart from
describing the model, our goal is part cautionary (it is not evident that general theories of synchronous or asynchronous networks
have much to contribute to stability problems involving structural change), and part comparative with the models we describe later for microgrids.

Assume a power grid with synchronous generators, DC power sources, transmission lines and various types of load.
We assume a reference frequency $\omega_R$ for the power grid, usually $50$Hz or $60$Hz, and note
that frequency synchronization is critical for the stability of the power grid: our equations will be 
written nominally in terms of phases $\theta_i(t)$ but for the models,  we can always replace $\theta_i(t)$ by $\theta_i(t)-\omega_R t$ to get the (same) equations 
for phase deviations that are needed for stability theory (phase differences, but not absolute phases, matter).

Formally, assume given an undirected (connected) weighted graph $\mc{G}$ with  node set $\mc{V} = \is{n}$ and edge
set $\mc{E} \subset \mc{V}^2$. Nodes will be partitioned as
$
\mc{V} = \mc{V}_1 \cup \mc{V}_2 \cup \mc{V}_3,
$
where $ \mc{V}_1$ consists of synchronous generators, $\mc{V}_2$ are DC power sources, and $\mc{V}_3$ 
comprises various types of load (see below and note we do not consider all types of load). 

Each edge $(i,j) \in \mc{E}$, $i \ne j$, is weighted by a non-zero admittance $Y_{ij}\in\cx$ and corresponds to a transmission line. The imaginary part $\mathfrak{I}(Y_{ij})$
is the susceptance of transmission line and $\mathfrak{R}(Y_{ij})$ is the conductance. 
Typically, a high voltage AC transmission line is regarded as lossless ($\mathfrak{R}(Y_{ij}) = 0$) and inductive ($\mathfrak{I}(Y_{ij}) > 0$).  We allow self-loops $i = j$, these will
correspond to loads modelled as impedances to ground (nonzero ``shunt admittances'').

To each node is associated a voltage phasor $V_i = |V_i|e^{\imath \theta_i}$ corresponding to phase $\theta_i$ and magnitude $|V_i|$ of the sinusoidal solution to the circuit equations.

For a lossless network, the power flow from node $i$ to node $j$ is given by
$a_{ij} \sin(\theta_i - \theta_j)$, where $a_{ij} = |V_i||V_j| \mathfrak{I}(Y_{ij})$ gives the maximal power flow
(see Kundur~\cite[Chapter 6]{Ku}). 

\subsubsection{Synchronous generators}
We assume dynamics of synchronous generators are given by
\begin{equation}
\label{sync}
M_i \theta_i'' + D_i \theta_i' = P_{m,i} + \sum_{j=1}^n a_{ij} \sin(\theta_j -\theta_i), \; i \in \mc{V}_1,
\end{equation}
where $\theta_i, \theta_i'$ are generator rotor angle and frequency, $M_i, D_i > 0$ 
are inertia and damping coefficients, and $P_{m,i}$ is mechanical power input.
\subsubsection{DC/AC inverters: droop controllers}
Each DC source in $\mc{V}_2$ is connected to the AC grid via a DC/AC inverter following a frequency droop control law which obeys
the dynamics~\cite{SPDB}
\begin{equation}
\label{dc}
D_i \theta_i' = P_{d,i} +  \sum_{j=1}^n a_{ij} \sin(\theta_j -\theta_i), \; i \in \mc{V}_2.
\end{equation}
\subsubsection{Frequency dependent loads}
We assume the active power demand drawn by load $i$ consists of a constant term $P_{l,i}>0$ and a frequency dependent term
$D_i \theta_i'$, $D_i > 0$, leading to the power balance equation
\begin{equation}
\label{fdl}
D_i \theta_i' = -P_{l,i}  +  \sum_{j=1}^n a_{ij} \sin(\theta_j -\theta_i), \; i \in \mc{V}_{3,f},
\end{equation}
where $\mc{V}_{3,f}$ is the subset of $\mc{V}_3$ consisting of frequency dependent loads. 
Equation~\Ref{fdl} is of the same form as \Ref{dc}, and
we may replace $\mc{V}_2$ by $\mc{V}_2 \cup \mc{V}_{3,f}$ and consider the general equation
\begin{equation}
\label{fdl1}
D_i \theta_i' = \omega_i  +  \sum_{j=1}^n a_{ij} \sin(\theta_j -\theta_i), \; i \in \mc{V}_{2},
\end{equation}
where $\omega_i$ is positive if the node is a DC generator and negative if it is a frequency dependent load.

We can similarly allow for loads which are synchronous motors, incorporate them in $\mc{V}_1$ and consider
\begin{equation}
\label{sync1}
M_i \theta_i'' + D_i \theta_i' = \omega_i + \sum_{j=1}^n a_{ij} \sin(\theta_j -\theta_i), \; i \in \mc{V}_1,
\end{equation}
where $\omega_i$ is positive  if the node is a synchronous generator and negative if it is a synchronous motor.

\subsubsection{Constant current and constant admittance loads}
We assume the remaining loads each require a constant amount of current and have a shunt admittance (to ground).
In this case we have a current balance equation and, through the process of Kron reduction~\cite{DBKron}, may obtain a reduced 
network the equations of which are
\begin{eqnarray}
\label{M1}
M_i \theta_i'' + D_i \theta_i' = \tilde{\omega}_i + \sum_{j=1}^n \tilde{a}_{ij} \sin(\theta_j -\theta_i+\varphi_{ij}), \; i \in \mc{V}_1,\\
\label{M2}
D_i \theta_i' = \tilde{\omega}_i  +  \sum_{j=1}^n \tilde{a}_{ij} \sin(\theta_j -\theta_i+\varphi_{ij}), \; i \in \mc{V}_{2}.
\end{eqnarray}
We refer to~\cite{DB} for the explicit form of the coefficients in (\ref{M1},\ref{M2}).

The original power grid network is typically sparse with many nodes -- $\mc{V}_3$ is large. The process of Kron reduction results in a much smaller
network which will be all-to-all coupled provided that the graph defined by $\mc{V}_3$ is connected~\cite{DBKron}. However, even if the original 
transmission lines are lossless, the phase shifts $\phi_{ij}$ will generally be non-zero and not necessarily always small 
(we refer to~\cite[\S 6.2 Figure 4]{Motter2015} for data from a real power grid network). The presence of phase shifts can and does 
make it harder to frequency synchronize (\ref{M1},\ref{M2}).

From the point of view of transmission line failure in a power grid, even if the removal of an edge still results in a all-to-all coupled reduced network,
many of the coupling coefficients $\tilde{a}_{ij}$ will change.  It is a hard problem, that goes beyond existing analytical theory 
for synchronous and asynchronous networks, to get good insight into whether or not a breakdown will destabilize the network 
(this is irrespective of phenomena like Braess's paradox~\cite{WT,PP}).

\subsubsection{Microgrids}
Assume given a stable power grid network, robust to ``small'' changes in power demand, and consider
the problem of modelling a microgrid and  its combination or separation from the main grid.
We outline structural and logical issues to make transparent
the connection with asynchronous networks and largely ignore dynamics so as to keep the model simple and our discussion short
(we refer to~\cite{DCB,SPDB,DSPB,BGG} for more details and references on microgrids and control from a large and rapidly developing literature in this area).
Assume power generation in the microgrid is from 
DC generators (such as solar power or DC wind power) and that $\mc{V}_1 = \emptyset$ (most motor loads  are not synchronous). 
Assume the microgrid is Kron reduced.

Unlike the power grid model described above, we
allow directed (one way) connections and a constraining node. Consider the simplified network $\Nn = \{\NS,N_B,N_G,N_P\}$, where
the nodes $N_B ,N_G,N_P$ correspond to a large capacity battery (buffer), a DC generator, and main power grid respectively, and define subnetworks $\Nn_M = \{N_B,N_G\}$ (microgrid) and
$\Nn_P = \{N_P\}$ (main power grid).

The battery acts as reserve storage or buffer for the microgrid; in particular to maintain power in the event of intermittent loss of
generated DC power or when the microgrid has been separated ``islanded'' from the main power grid. We suppose battery capacity $B= B(t) \in [0,B_M]$, where $B_M$ corresponds to the battery
being fully charged. We suppose that the DC generator produces power $O = O(t) \in [0,O_M]$, where $O_M$ is the maximum power than can be generated.

The constraining node will play a role when the microgrid is islanded and is to be reconnected to the main power grid: either because the
microgrid has insufficient power for the microgrid loads or because the microgrid has an excess of available power some of which can now be contributed to the 
main power grid. In either case a transition process needs to be implemented where the droop controller for the DC/AC converter needs to bring the AC 
output of the microgrid in precise voltage (phase, frequency and magnitude) synchronization with the state of the power grid at the connection point(s) to the microgrid. 
Similarly, we can constrain when the microgrid is to be islanded from the main grid so that the reduction in power contributed to the main power grid is gradual
and done over an appropriate time scale so as not to destabilize the main power grid. 

Leaving aside the dynamics of islanding and combining the microgrid with the main power grid, the generalized connection structures and control logic we need for
management of the microgrid are complex and depend on several thresholds which may need to be time dependent -- for example, if we
use a time dependent model for the projected microgrid power load. If the microgrid is islanded, we work with $\Nn_M$ and use the generalized connection structure
\[
\A_M = \{\alpha = N_G \arr N_B, \beta=N_B \arr N_G, \ECS\}.
\]
The connection structure $\alpha$ corresponds to the DC generator having sufficient output to supply all power needed for the microgrid load and with a surplus which can be used to 
charge the battery, $\beta$ corresponds to battery and generator providing all necessary power for the microgrid, and
$\ECS$ corresponds to the generator providing all needed power for the microgrid and either there is surplus power available for battery charging or the battery is fully charged.
Thresholds that determine switching between these states are chosen so as to avoid ``chattering'' in the control system. 

If the microgrid is combined with the main power grid, this can be either because battery and DC generators cannot provide sufficient power for the microgrid load or because
the microgrid has surplus power which can be contributed to the main power grid or because the main power grid is stressed (possibly locally detected by
frequency variation) and the battery state of the microgrid is sufficiently high to allow a temporary power contribution to the main grid. As generalized connection structure $\A$
we take the set of connection structures
\[
N_G \arr N_M,\; N_G \arr N_M\leftarrow N_B,
\]
\[
N_B \arr N_G\leftarrow N_M,\; N_M \arr N_G,\;
N_M\arr N_G\arr N_B, 
\]
Each of these connection structures has a natural interpretation. For example, $N_M\arr N_G\arr N_B$ corresponds to the main power grid contributing to both the load of the
microgrid and battery charging while $N_G \arr N_M\leftarrow N_B$ means battery and DC generator are contributing power to the main power grid as well as supplying all the power for the microgrid. 
On the other hand, $N_G \arr N_M$ means DC generated power, but not battery power, is being contributed to the main power grid. 

Of course, what we have described above is highly simplified as we have taken no account of (1) multiple DC generators and batteries within a microgrid, or
(2) multiple microgrids. In the latter case, we need to take care that microgrid switching does not synchronize as this could lead to large destabilizing
changes in load on the main grid.

%% file: sectionI_06.tex
\section{Products of asynchronous networks}
\label{sec:Products}
We conclude part I with the definition of the product of asynchronous networks and give 
sufficient conditions for an asynchronous network to decompose as a product of two or more
asynchronous networks. Although the methods we use are elementary, the study of products is 
illuminating as it clarifies some subtleties in both the event map and the functional structure that
are not present in the theory of synchronous networks. These ideas play a central role in the
proof of the modularization of dynamics theorem in part II.

\subsection{Products}
\label{sec:prod}

Given $\aa,\beta \in M(k)$, define $\aa \vee \beta\in M(k)$ (the join 
of~$\aa$ and~$\beta$) by
\[
(\aa \vee \beta)_{ij} = \max \sset{\aa_{ij},\beta_{ij}}, \;i,j \in\is{k}
\]
(the max-plus addition of tropical algebra~\cite{HOW}). We have 
$\aa \vee \ECS = \aa$ for all $\aa\in M(k)$. If 
$\A, \mathcal{B}\subset M(k)$ are generalized connection structures,
define the generalized connection structure $\A \vee \mathcal{B}$ by
\[
\A \vee \mathcal{B} = \set{\aa \vee \beta}{\aa \in \A,\; \beta \in \mathcal{B}}. 
\]
Note that $\ECS \in \A \vee \mathcal{B}$ if and only if 
$\ECS \in \A\cap\mathcal{B}$.
Consequently, if $\ECS \in \A\vee\mathcal{B}$, then $\A , \mathcal{B} \subset \A \vee \mathcal{B}$.

Suppose that~$A$ is a nonempty subset of $\is{k}$ containing $k_A$ 
elements. There is a natural embedding of $M(k_A)$ in $M(k)$ defined 
by relabelling the matrices in $M(k_A)$ according to~$A$. Specifically, 
map the matrix $(\aa_{ij})_{i,j \in A} \in M(k_A)$ to the matrix 
$\ah\in M(k)$ defined by
\[
\ah_{ij} = \begin{cases}
\aa_{ij} & \text{for }i, j \in A,\\
0 & \text{otherwise}.
\end{cases}
\]
This embedding extends to an embedding of $\CC{k_A}$ in 
$\CC{k}$ by 
\[
\ah_{i0} = \begin{cases}
\aa_{i0} & \text{for }i  \in A, \\
0 & \text{otherwise.}
\end{cases}
\]
Given disjoint nonempty subsets $A, B$ of~$k$, regard $\CC{k_A}, \CC{k_B}$ 
as embedded in $\CC{k}$. Given $\aa \in \CC{k_A}$, $\beta\in \CC{k_B}$, 
define
\[\aa \vee \beta \defo \ah\vee\bh\in\CC{k}.\]
This extends to the join $\A \vee \mathcal{B}$ of generalized 
connection structures on disjoint sets of nontrivial nodes.

Let $\Nn = \{N_0, \dotsc,N_k\}$ and $A$ be a proper subset of~$\is{k}$.
Define $\mathcal{N}^A = \{N_j \dd j \in \bu{A}\}$ and $\Mb_A = \prod_{i \in A} M_i$.
Denote points in $\Mb_A$ by $\XX_A$. Suppose $B = \is{k}\sm A$. We have 
 $\Nn^{A}\cap \Nn^{B}=\sset{\NS}$ and $\Mb_A \times \Mb_B \approx \Mb$.
If $\is{C}^A$, $\is{C}^B$ are constraint structures on
$\mathcal{N}^A$, $\mathcal{N}^B$ respectively, let $\is{C} = \is{C}^A \vee \is{C}^B$ denote the induced constraint structure
on $\Nn$ -- well defined since constraints depend only on nodes and $A\cap B = \emptyset$.

More generally, given disjoint node sets $\mathcal{N}^A = \{N_j \dd j \in \bu{A}\}$, $\mathcal{N}^B = \{N_j \dd j \in \bu{B}\}$,
we can identify $A,B$ with complementary subsets of $\is{k}$, where $k$ is the total number of elements in $A \cup B$, and then follow the
conventions described above.

\begin{Def}
\label{prodasy}
(Notation and assumptions as above.)
Given asynchronous networks 
$\mathfrak{N}^X = (\Nn^X,\A^X,\mathcal{F}^X,\mathcal{E}^X)$, $X\in\sset{A,B}$, 
define the product 
$\mathfrak{N}^A \times \mathfrak{N}^{B}$ to
be the asynchronous network $\Net =(\Nn,\A,\mathcal{F},\mathcal{E})$ where
\begin{enumerate}
\item $\mathcal{N} = \mathcal{N}^A \cup \mathcal{N}^B$,
\item $\is{C} = \is{C}^A \vee \is{C}^B$,
\item $\A = \A^A \vee \A^B$,
\item $\mathcal{F} = \mathcal{F}^A \times \mathcal{F}^B = \{ \f_A^\aa\times \f_{B}^\beta\dd \aa\in\A^A,\; \beta\in \A^B\}$, and
\item $\Ee$ is defined by $$\mathcal{E}(\XX_A,\XX_B) = \mathcal{E}^A(\XX_A) \vee  \mathcal{E}^B(\XX_B),\; \text{for } (\XX_A,\XX_B) \in \Mb_A\times \Mb_B.$$
\end{enumerate}
\end{Def}

\begin{rem}
If  $\mathfrak{N}^A,  \mathfrak{N}^B$ are proper (or amenable), then so is $\Net^A\times\Net^B$.
\end{rem}
\begin{lemma}
(Notation of definition~\ref{prodasy}.)  
The network vector field on $\mathfrak{N}^A\times\mathfrak{N}^B$ is given by
\begin{equation}
\label{DEQD}
\is{F}(\XX_A,\XX_B) = (\is{f}_A^{\mathcal{E}^A(\XX_A)}(\XX_A),\is{f}_B^{\mathcal{E}^B(\XX_B)}(\XX_B)),
\end{equation}
for all $(\XX_A,\XX_B) \in \Mb_A\times  \Mb_B$.
\end{lemma}
\begin{proof}Immediate from the definitions. \end{proof}

\subsection{Decomposability}
\label{sec:Indecomposability}

\begin{Def}
\label{indec}
An asynchronous network $(\Nn,\A,\mathcal{F},\mathcal{E})$
is \emph{decomposable} if it can be written as a product of asynchronous
networks. If the network is not decomposable, it is \emph{indecomposable}.
\end{Def}
\begin{exam}
\label{syncexample}
Suppose that $\Nn$ is a synchronous network with connection structure $\alpha \in M(k)$ and
$\aa$-admissible network vector field $\f$ satisfying conditions (N1--3) of section~\ref{generalities}. Since $\alpha$ encodes the dependencies of $\f$
it is trivial that $\Nn$ can be written as a product of two synchronous networks iff the network graph $\Gamma_\aa$ is disconnected.
\examend
\end{exam}

Our aim to find sufficient conditions on an asynchronous network for it to be decomposable.
\begin{Def}
The \emph{connection graph} of the asynchronous network 
$\Net = (\Nn,\A,\F,\Ee)$ is the graph defined by the \ZO matrix 
$\Gamma_\Net\defo \bigvee_{\aa\in\A}\af$.
\end{Def}

\begin{lemma}
\label{lem:NecDecomp}
If an asynchronous network $\Net$ is decomposable,
then the {connection graph}
$\Gamma_\Net$ of $\Net$
has at least two connected components.
\end{lemma}

\begin{proof} If $\mathfrak{N}$ is decomposable, then 
$\mathfrak{N} = \mathfrak{N}^A\times\mathfrak{N}^B$, where $A,B$ are proper complementary subsets of $\is{k}$. 
Since there are no connections between nodes in $\mathcal{N}^A$ and $\mathcal{N}^B$,
$\Gamma_\Net$ has at least two connected components.
\end{proof}
\begin{rem}\label{dep}
Lemma~\ref{lem:NecDecomp} gives a necessary condition for 
decomposability which is not sufficient. 
There are two issues.
First, the event map encodes information about spatial dependence of node interactions that cannot be 
deduced from the connection graph. Second, the admissible vector fields may have dependencies that are
incompatible with decomposability. 
\end{rem}

\begin{exam}\label{ex:StrucIndecomp}
Let $k = 2$, $M_1 = M_2 = \real$. Define connection structures $\aa_i = \NS \arr N_i$, $i \in \is{2}$ and
generalized connection structure 
$\A = \{\ECS, \aa_1, \aa_2,\beta = \aa_1 \vee \aa_2\}$. 
Suppose the event map is given by
\[
\Ee(x_1,x_2) = \begin{cases}
\aa_1,& \text{if } x_1 < 0, x_2 =0\\
\aa_2, & \text{if } x_1 = 0, x_2 >0\\ 
\beta, & \text{if } x_1 = x_2 = 0,\\ 
\ECS, & \text{otherwise}
\end{cases}
\]
In this case, $\A = \A^1\vee\A^2$, where $\A^i = \{\ECS,\aa_i\}$, $ i \in \is{2}$,
and the network graph is disconnected. However, 
there is no way to write $\Ee(x_1, x_2)$ as $\Ee^1(x_1)\vee\Ee^2(x_2)$ as  
the event sets involving $x_1\in M_1$ depend nontrivially on~$x_2\in M_2$.
Hence the network cannot be decomposable or 
even equivalent to a decomposable network whatever choice we make for admissible vector fields.

Suppose instead we define the event map by
\[\tilde\Ee(x_1,x_2) = \begin{cases}
\aa_1, & \text{if } x_1 = 0, x_2\neq 0\\
\aa_2, & \text{if } x_2 = 0, x_1\neq 0\\ 
\beta,&  \text{if } x_1 = x_2 = 0\\
\ECS, & \text{otherwise}
\end{cases}
\]
In this case  $\A = \A^1\vee\A^2$ and we may write $\Ee = \Ee^1 \vee \Ee^2$ where
$\Ee^i(0) = \aa_i$, and $\Ee^i(x_i) = \ECS$, $x_i\ne 0$, $i \in \is{2}$. 
Suppose that $\f^{\aa_1}(x_1,x_2) = (0,v_2)$, $\f^{\aa_2}(x_1,x_2) = (v_1,0)$, $\f^\ECS(x_1,x_2) = (v_1,v_2)$, where $v_1,v_2 \ne 0$. For the
moment leave $\f^\beta$ unspecified.
Define $\F^i=\{\f_i^\ECS, \f_i^{\alpha_i}\}$, where $\f_i^\ECS(x_i) = v_i$, $\f_i^{\aa_i}(x_i) = 0$, $i \in \is{2}$. Observe that
$\f^\ECS = \f_1^\ECS \times \f_2^\ECS$, $\f^{\aa_1} = \f_1^{\aa_1} \times \f_2^\ECS$ and $\f^{\aa_2} = \f_1^\ECS\times \f_2^{\aa_2}$.
For $(\Nn,\A,\F,\Ee)$ to be a product we additionally
require $\f^\beta(x_1,x_2) = (\f_1^{\aa_1}(x_1),\f_2^{\aa_2}(x_2)) = (0,0)$, all $(x_1,x_2) \in \real^2$.
In particular, if $\f^\beta(0,0) \ne (0,0)$, the network $(\Nn,\A,\F,\Ee)$ is not even equivalent to a product network.
However, if $\f^\beta(0,0) = (0,0)$, then the network $(\Nn,\A,\F,\Ee)$ will be equivalent to a product network
if we redefine $\f^\beta$ to be $\f_1^{\aa_1} \times \f_2^{\aa_2}$ (this does not change the values of $\f^\beta$ on $E^\beta$).
\examend
\end{exam}

\subsection{Sufficient conditions for decomposability}
\label{sec:Decomposability}

Let~$\Net$ be an asynchronous network with~$k$ nodes and $C$ be 
a proper connected component of the connection 
graph~$\Gamma_\Net$. Identify $C$ with
the nonempty subset of~$\is{k}$ corresponding to the labels of the
nodes in the component $C$. Let~$\ol{C} = \is{k}\sm C$.
Since~$C$ is a connected component of $\Gamma_\Net$, 
we can write each $\aa\in \A$ uniquely as 
$\aa=\aa_C\vee\aa_{\ol{C}}$, where $\aa_C$, 
$\aa_{\ol{C}}$ are connection structures on $\Nn^C$ and 
$\Nn^{\ol{C}}$ respectively. Set 
$\A^C = \{\aa_C \dd \aa \in\A\}$. We have a well defined 
projection $\pi_C:\A\arr \A^C$ defined by $\pi_{C}(\aa)=\aa_C$.

Define the event map $\mathcal{E}^C: \Mb_C\times\Mb_{\ol{C}}\to\A^C$ 
by
\[
\mathcal{E}^C(\XX_C, \XX_{\overline{C}}) = \pi_C(\mathcal{E}(\XX_C, \XX_{\overline{C}})).
\]

\begin{Def}
An asynchronous network $\Net$ 
is \emph{structurally decomposable} if for 
any connected component~$C$ of the connection graph $\Gamma_\Net$, 
the map $\Ee^C$ is independent of 
$\XX_{\ol{C}}\in\Mb_{\ol{C}}$
(that is, $\Ee^C(\XX_C, \XX_{\ol{C}}) = \Ee^1(\XX_C)$ where $\Ee^1: \Mb_C \to \A^C$).
\end{Def}

\begin{rem}
Structural decomposability implies conditions on 
structural dependencies that will generally be different from
the dependencies of the network vector field. For example, suppose a component $C$ of the connection graph contains the node $N_1$. If the node $N_1$ is stopped there may 
be a condition that $N_1$
will restart when the state of another node, say $N_2$, attains a certain
value. Necessarily, $N_2$ must lie in $C$ (structural decomposability). However, there need be no connection between $N_1$ and $N_2$ unless $C$ contains exactly two nodes. 
\end{rem}
Suppose that $\Net$ is structurally decomposable and that $\Gamma_\Net$ has connected components $C_1,\dotsc,C_q$.
Set $\Mb_\ell = \Mb_{C_\ell}$, $\A^\ell = \pi_{C_\ell}(\A)$, $\ell \in \is{q}$.
By structural decomposability we may write $\Ee(\XX) = \bigvee_{\ell\in \is{q}} \Ee^\ell(\XX_\ell)$ where
$\Ee^\ell:\Mb_\ell \arr \A^\ell$. For $\alpha \in \A$, $\ell \in \is{q}$, set $\alpha_\ell = \pi_{C_\ell}(\alpha) \in \A^\ell$ and $E^\ell_{\alpha_\ell} = (\Ee^\ell)^{-1}(\alpha_\ell)\subset \Mb_\ell$.

\begin{lemma}
(Notation as above.)
If $\Net$ is structurally decomposable and $\Gamma_\Net$ has connected components $C_1,\dotsc,C_q$, then 
\[
E_\alpha = \prod_{\ell \in \is{q}} E^\ell_{\aa_\ell} \subset \prod_{\ell \in \is{q}} \Mb_\ell,\;\text{for all } \aa\in\A.
\]
\end{lemma}
\proof An immediate consequence of structural decomposability. \qed

If $C$ be a proper connected component of the connection graph $\Gamma_\Net$ of an asynchronous network $\Net$, then
by admissibility 
\[
\f^\aa = \f_C^\aa\times \f_{\ol{C}}^\aa, \; \text{for all } \alpha \in \A,
\] where 
$\f^\aa_C: \Mb_C\to T\Mb_C$ and 
$\f^\aa_{\ol{C}}: \Mb_{\ol{C}}\to T\Mb_{\ol{C}}$.

In order that $\Net$ 
be decomposable, this decomposition has to be compatible 
with the projections 
$\pi_C: \A\to \A^C$, $\pi_{\ol{C}}: \A\to \A^{\ol{C}}$. In particular, if
connections in the set of nodes that are in~$\ol{C}$ are added or deleted,
dynamics on $\Mb_C$ is not affected. 

\begin{Def}
(Notation as above.)
The asynchronous network $\Net$ is \emph{dynamically decomposable} if 
for any connected component~$C$ of $\Gamma_\Net$, we have
\[
\f_C^\aa = \f_C^\beta
\]
for all $\alpha,\beta \in \A$ such that $\pi_C(\alpha) = \pi_C(\beta)$.
\end{Def}

\begin{lemma}
(Notation as above.)
Input consistent asynchronous networks are dynamically decomposable.  In particular, 
asynchronous networks with additive input structure are dynamically decomposable.
\end{lemma}

\begin{proof}
Given $i \in \is{k}, \alpha \in \A$, let $J(i, \aa)$ be the associated dependency set for node~$N_i$.
If $\alpha,\beta \in \A$ and $J(i, \aa) = J(i, \beta)$, then $f_i^\aa = f_i^\beta$ by
input consistency.
If $i\in C$, where~$C$ is a connected component of the network 
graph~$\Gamma_\Net$, then $J(i, \aa)\cap\is{k}\subset C$ for 
all~$\aa\in \A$. Hence 
$J(i, \aa)=J(i, \aa_C\vee \aa_{\ol{C}})$ is independent of $\aa_{\ol{C}}$.
Input consistency implies that 
$f_i^{\aa_C\vee\beta} = f_i^{\aa_C\vee\gamma}$ for all
$\beta, \gamma\in\A_{\ol{C}}$ which yields dynamical decomposability.
\end{proof}

We now state the main result of this section.

\begin{thm}\label{lem:SufDecomp}
Let $\Net$ be a structurally and dynamically 
decomposable asynchronous network with connection graph~$\Gamma$. 
If~$\Gamma$ has connected components $C_1, \dotsc, C_q$ then there 
exist indecomposable asynchronous networks
$\Net^1, \dotsc, \Net^q$ such that
\[\Net = \Net^1\times\dotsb\times\Net^q.\]
\end{thm}

\begin{proof}
For $\ell \in \is{q}$, define $\A^\ell = \{\alpha_\ell \defoo \pi_\ell(\aa)\dd\aa\in\A\}$
and $\F^\ell = \{\f_\ell^{\aa_\ell} \defoo \f_{C_\ell}^\alpha:\Mb_\ell \arr T\Mb_\ell \dd \aa \in \A\}$. By dynamical indecomposability
we have $\f^\alpha = \prod_{\ell \in \is{q}} \f_\ell^{\aa_\ell}$, for all $\alpha \in \A$.
Constraint structures are
defined for individual nodes and so factorise naturally.  Let $\Ee^\ell:\Mb_\ell\arr \A^\ell$ be the event maps given by structural indecomposability.
If we let $\Net^\ell$ be the asynchronous network $(\Nn^\ell,\A^\ell,\F^\ell,\Ee^\ell)$, where $\Nn^\ell = \{\NS\}\cup \{N_i \dd i \in C_i\}$, $\ell\in\is{q}$, then
$\Net=\prod_{\in\is{q}} \Net^i$.
\end{proof}

Our concluding result on decomposability is an immediate consequence of lemma~\ref{lem:NecDecomp} and theorem~\ref{lem:SufDecomp}. 

\begin{cor}\label{prop:DecompConnectionGraph}
A structurally and dynamically decomposable asynchronous 
network~$\Net$ is decomposable if and only if its connection 
graph has more than one nontrivial connected component.
\end{cor}

\subsection{Factorization of asynchronous networks}

Assume for this section that $\Net = (\Nn,\A,\F,\Ee)$ is an asynchronous network which is not necessarily
structurally or dynamically indecomposable.
\begin{Def}
The asynchronous network $\Net^1$ is a factor of $\Net$ if there is an asynchronous network
$\Net^2$ such that $\Net = \Net^1 \times \Net^2$. 
\end{Def}
The proof of the next lemma is immediate from the definition of a product.
\begin{lemma}\label{neccond}
If $\Net^1$ is a factor of $\Net$, then
the connection graph $\Gamma_{\Net^1}$ is a union of connected components of $\Gamma_{\Net}$.
\end{lemma}
\begin{rem}
If $\Net^1$ is indecomposable, the connection graph $\Gamma_{\Net^1}$ may have more than one component -- unless
$\Net$ is structurally and dynamically indecomposable (theorem~\ref{lem:SufDecomp}).
\end{rem}

\begin{prop}
\label{uniqfactor}
Every asynchronous network $\Net$ has a factorization $\prod_{a\in\is{q}}\Net^a$ as a product of indecomposable
asynchronous networks. The factorization is unique, up to the order of factors.
\end{prop}
\proof Existence is obvious. The uniqueness of factorization follows easily from lemma~\ref{neccond}. \qed

\vspace*{0.1in}

\noindent {\bf Acknowledgements.}
CB would like to thank Marc Timme at the Max Planck Institute for Dynamics
and Self-Organization for continuing hospitality. MF would like to
thank the Mathematics Department at Rice University, where
much of this work was done, for providing a warm and supportive working environment, and
Steve Furber of Manchester University for his penetrating insights and questions. 
Both authors would like to acknowledge fruitful discussions with 
many people --- too many to name individually here --- whose input has proved 
invaluable in helping to shape these ideas.

%% file: AsynchronousNetworks_I.bbl
\begin{thebibliography}{99}
\bibitem{AF1} N Agarwal \& M Field. `Dynamical equivalence of networks of coupled dynamical
systems I: asymmetric inputs', \emph{Nonlinearity} {\bf 23} (2010), 1245--1268.
\bibitem{AF2} N Agarwal \& M Field. `Dynamical equivalence of networks of
coupled dynamical systems II: general case', \emph{Nonlinearity}, {\bf 23} (2010), 1269--1289.
\bibitem{AADF} M Aguiar, A P S Dias, P Ashwin, \& M Field. `Dynamics of coupled cell networks: synchrony, heteroclinic cycles and inflation',
\emph{Journal of Nonlinear Science}, {\bf 21}(2) (2011), 271--323.
\bibitem{Alon} U Alon. \emph{An Introduction to Systems Biology. Design Principles of Biological Circuits} (Chapman \& Hall/CRC, Boca Raton, 2007).
\bibitem{Anderson1972} P W Anderson. `More is different', \emph{Science} {\bf 177}(4047) (1972), 393--396. 
\bibitem{As} D V Anosov. `Stability of equilibrium positions in relay systems', \emph{Automation Remote Control} {\bf 20}
(1959), 135--149.
\bibitem{BA} A L Barab'asi and R Albert. `Statistical mechanics of complex networks', \emph{Rev. Mod Phy.} {\bf 74} (2002), 47--94.
\bibitem{BBH} V  Belykh, I  Belykh, and M  Hasler. `Connection graph stability method for synchronized coupled chaotic systems', \emph{Physica D} {\bf 195}(1-2) (2004), 159--187.
\bibitem{BBKP} I Belykh, M di Bernardo, J Kurths, \& M Porfri. `Evolving
dynamical networks', \emph{Physica D}, {\bf 267} (2014), 1--6. 
\bibitem{bbc} M di Bernardo, C J Budd, \& A R Champneys, `Normal form maps for grazing bifurcations inn-dimensional piecewise-smooth dynamical systems',
\emph{Physica D}  {\bf 160} (2001), 222-254.
\bibitem{BBCK} M di Bernardo, C J Budd, A R Champneys, \& P Kowalczyk. \emph{Piecewise-smooth Dynamical Systems} (Springer, Applied Mathematical Science 163, London, 2008).
\bibitem{FA} C Bick. `Local representation of asynchronous networks by Filippov systems', in preparation.
\bibitem{BF1} C Bick \& M J Field. `Asynchronous networks: modularization of dynamics theorem', \emph{Nonlinearity} {\bf 30} (2) (2017), 595--621.
\bibitem{BF2} C Bick \& M J Field. `Asynchronous networks: stochastic event processes and local times', in preparation.
\bibitem{Boc} J Bochnak, M Coste, \& M-F Roy. \emph{Real algebraic geometry} (Springer, 1998).
\bibitem{BGG} M Bortolini, M Gaberi, and A Graziani. `Technical and Economic Design of photovoltaic and bettery energy storage system',\emph{Energy Conversion and Management} {\bf 86} (2014), 81--92.
\bibitem{CD} N Caporale \& Y Dan. `Spike timing dependent plasticity: a Hebbian learning rule' \emph{Annu. Rev. Neurosci.} {\bf 31} (2008), 25--36.
\bibitem{CH} R Cohen and S Havin. \emph{Complex Networks: Structure, Robustness and Function} (Cambridge University Press, 2010).
\bibitem{PNet} R David \& H Alla. \emph{Discrete, Continuous, and Hybrid Petri Nets} (Springer-Verlag, 2010).
\bibitem{DT} F D\"orfler. `Dynamics and Control in Power Grids and Complex Oscillator Networks', PhD thesis, Mechanical Engineering Department, University of California at Santa Barbara, 2013.
\bibitem{DB} F D\"orfler \& F Bullo. `Synchronization and transient stability in
power networks and non-uniform Kuramoto oscillators', \emph{SIAM Journal on Control
and Optimization} {\bf 50} (3) (2012), 1616-1642.
\bibitem{DBKron} F D\"orfler \& F Bullo. `Kron reduction of graphs with applications to electrical networks', \emph{IEEE Trans. Circuit Syst.} 
{\bf 160}(1) (2013), 150--163. 
\bibitem{DCB} F D\"orfler, M Chertkov,  \& F Bullo. `Synchronization in Complex Oscillator Networks and Smart Grids',
\emph{PNAS} {\bf 110}(6) (2013), 2005--2010.
\bibitem{DSPB} F D\"orfler, J W Simpson-Porco, \& F Bullo. `Breaking the Hierarchy: Distributed Control \& Economic Optimality in Microgrids', 
\emph{IEEE Trans. on Control of Network Systems} (2016), to appear. 
\bibitem{EK} G B Ermentrout \& N Kopell. `Multiple pulse interactions and averaging in systems of coupled neural oscillators',
\emph{J. Math. Biol.} {\bf 29} (1991), 193--217.
\bibitem{F2014} M J Field. `Heteroclinic networks in homogeneous and heterogeneous identical cell systems', \emph{J. Nonlinear Science} {\bf 25}(3) (2015), 779--813.
\bibitem{Fi} A F Filippov. \emph{Differential Equations with Discontinuous Righthand Sides} (Kluwer Academic Publishers, 1988). 
\bibitem{FNP} G Filatrella, A H Nielsen, \& N F Pederson. `Analysis of a power grid using a Kuramoto-like model',
\emph{Eur. Phys. J. B} {\bf 61} (2008), 485--491.
\bibitem{Ge} W Gerstner, R Kempter, L J van Hemmen \& H Wagner. `A neuronal learning rule for sub-millisecond temporal coding', \emph{Nature} {\bf 383}
(1996), 76--78.
\bibitem{GK} W Gerstner \& W Kistler. \emph{Spiking Neuron Models} (Cambridge University Press, 2002).
\bibitem{Gibson} C Gibson, K Wirthm\"uller, A A du Plessis, \& E Looijenga. \emph{Topological stability of smooth mappings}
(Springer Lect. Notes in Math. {\bf 552}, 1976).
\bibitem{GST} M Golubitsky, I Stewart, \& A T\"or\"ok.
`Patterns of Synchrony in Coupled Cell Networks with Multiple Arrows', \emph{SIAM J. Appl. Dynam. Sys.}
{\bf 4}(1) (2005), 78--100.
\bibitem{GBG} T E Gorochowski, M  Di Bernardo, \& C S Grierson. `Evolving
dynamical networks: A formalism for describing complex systems', \emph{Complexity} {\bf 17}(3) (2012), 18--25.
\bibitem{Haddad} W M Haddad, V Chellaboina, \& Sergey G Nersesov. \emph{Impulsive and Hybrid Dynamical Systems:
Stability, Dissipativity, and Control} (Princeton University Press, 2006).
\bibitem{HOW} B Heidergott, G J Olsder \& J van der Woude. \emph{Max Plus at Work} (Princeton Series in Applied Mathematics, 2006).
\bibitem{Holl} J H Holland. `Studying complex adaptive systems', \emph{J. Syst. Sci \& Complexity} {\bf 19} (2006), 1--8.
\bibitem{Holme} P Holme \& J Saram\"aki. `Temporal Networks', \emph{Physics Reports}, {\bf 519}(3), 97--125.
\bibitem{HI} F C Hoppensteadt and E M Izhikevich. \emph{Weakly Connected Neural Networks} (Springer, Applied Mathematical Science
{\bf 126}, Springer-Verlag, 1997).
\bibitem{KU} N Kashtan \& U Alon. `Spontaneous evolution of modularity and network motifs',
\emph{PNAS} {\bf 102} (39) (2005), 13773--13778.
\bibitem{Ku} P Kundur. \emph{Power System Stability and Control} (McGraw-Hill Inc, New York, 1994).
\bibitem{Ku2} IEEE/CIGRE Joint Task Force on Stability Terms and Definitions. \emph{Definition and Classification of Power System Stability},
\emph{IEEE Trans. Power Syst.} {\bf 19}(2) (2004), 1387--1401.
\bibitem{K1984} Y Kuramoto. \emph{Chemical Oscillations, Waves, and Turbulence} (Springer, Berlin, 1984).
\bibitem{KRG} Yu A Kuznetsov, S Rinaldi, \& A Gragnani. `One-parameter bifurcations in planar Filippov systems',
\emph{Int. J. of Bifur. and Chaos} {\bf 13}(8) (2003), 2157--2188.
\bibitem{LLK} J Ladyman, J Lambert, and K Wiesner. `What is a complex system?', \emph{Eur. J. for Phil. of Sc.} {\bf 3}(1) (2013), 33--67.
\bibitem{LL}  L Lamport. `Time, Clocks, and the Ordering of Events in a Distributed System', \emph{Comm. of the ACM}, {\bf 21}(7) (1978), 558--565.
\bibitem{BL}  H B Lawson. `Foliations', \emph{Bull. AMS} {\bf 80}(3) (1974), 369--418.
\bibitem{LuAtayJost} W Lu, F M Atay, \& J Jost. `Chaos synchronization in networks of coupled maps with
time-varying topologies', \emph{Eur. Phys. J. B} {\bf 63} (2008), 399--406. 
\bibitem{Mather} J N Mather. `Stratifications and mappings', \emph{Proceedings
of the dynamical systems conference, Salvador, Brazil} (ed M Peixoto) (Academic Press, New York, 1973), 195--223.
\bibitem{RMc} R S MacKay. `Nonlinearity in complexity science', \emph{Nonlinearity} {\bf 21}(12) (2008), 273--281.
\bibitem{Min}  N Minorsky. \emph{Theory of nonlinear control systems} (McGraw-Hill, New York, 1969).  
\bibitem{MDG} A Morrison, M Diesmann \& W Gerstner. `Phenomenological models of synaptic plasticity based on spike timing',
\emph{Biol. Cyber.} {\bf 98} (2008), 459--478.
\bibitem{Newman1} M E J  Newman. \emph{Networks: An Introduction} (Oxford University Press, 2010).
\bibitem{Motter2015} T Nishikawa and A E Motter. `Comparative analysis of existing models for power grid synchronization', \emph{New J. of Physics} {\bf 17} (2015).
\bibitem{PP} J P Pade and T Pereira. `Improving Network Structure can lead to Functional Failures', \emph{Sci. Rep.} {\bf 5} 9968 (2015), doi: 10.1038/srep09968.
\bibitem{PC} L M  Pecora and T L  Carroll. `Master stability functions for synchronized coupled systems',\emph{Phys. Rev. Lett.} {\bf 80} (1998), 2109--2112.
\bibitem{PN} A Pogromsky and H Nijmeijer. `Cooperative Oscillatory Behavior of Mutually Coupled Dynamical Systems', \emph{IEEE Trans. on circuit and systems-I: Fundamental theory and applications} {\bf 48}(2) (2001), 152--162.
\bibitem{PG} M A Porter and J P Gleeson. \emph{Dynamical Systems on Networks: A Tutorial} (Springer, Frontiers in Applied Dynamical Systems:
Reviews and Tutorials, Vol. 4, 2016).
\bibitem{RNL} Y Roudi, S Nirenberg, and P E Latham. `Pairwise Maximum Entropy Models for Studying Large
Biological Systems: When They Can Work and When They Can't', \emph{PLoS Computational Biology} {\bf 5}(5) (2009).
\bibitem{Sc} E Schr\"odinger. \emph{What is Life?} (Cambridge University Press, New York, 2012)
\bibitem{SPDB} J W Simpson-Porco, F D\"orfler, \& F Bullo. `Synchronization and power-sharing for 
droop controlled inverters in islanded microgrids', \emph{Automatica} {\bf 49}(9) (2013), 2603--2611.
\bibitem{TvS} A S Tanenbaum \& M van Steen. \emph{Distributed Systems: Principles and Paradigms} (Prentice Hall, 2nd edition, 2006).
\bibitem{T} S H Tindemans, V Trovato, and G Strbac. `Decentralized Control of Thermostatic Loads for Flexible Demand Response',
\emph{IEEE Transactions on Control Systems Technology} {\bf 23}(5) (2015), 1685--1700.
\bibitem{Whitney34} H Whitney. `Analytic extensions of differentiable functions defined on closed subsets', \emph{Trans AMS} {\bf 36} (1934), 63--89.
\bibitem{WT} D Witthaut \& M Timme. `Braess's paradox in oscillator networks, desynchronization and power outage',
\emph{New J. of Physics} {\bf 14} (2012), 083036.
\end{thebibliography}
